\documentclass{amsart}

\usepackage[shortlabels]{enumitem}
\usepackage{subfigure}
\usepackage{hyperref}

\usepackage{multirow}
\usepackage{multicol}

\usepackage{graphicx}
\usepackage[sort&compress,square,comma,numbers]{natbib}

\usepackage{amsthm,amssymb,amsmath}
\usepackage{hyperref}

\def\aa{{\ensuremath\boldsymbol a}}

\def\bb{{\ensuremath\boldsymbol b}}
\def\gg{{\ensuremath\boldsymbol g}}
\def\x{{\ensuremath\boldsymbol x}}

\def\w{{\ensuremath\boldsymbol w}}
\def\t{{\ensuremath\boldsymbol t}}

\def\domn{\mathcal D}
\def\Rnp{\mathbb R_+^n}
\def\epig{\textrm{epi}}
\def\hypg{\textrm{hyp}}
\def\conv{\textrm{conv}}

\newtheorem{defn}{Definition}
\newtheorem{lemma}{Lemma}
\newtheorem{prop}{Proposition}
\newtheorem{corollary}{Corollary}

\newtheorem{thm}{Theorem}

\title{Convex envelopes of bounded monomials on two-variable cones}

\author{Pietro Belotti}
\thanks{Dipartimento di Elettronica, Informazione e Bioingegneria, Politecnico di Milano\\
Orcid: \href{https://orcid.org/0000-0001-6591-6886}{0000-0001-6591-6886}\\
Email: \href{mailto:pietro.belotti@polimi.it}{pietro.belotti@polimi.it}}

\date{1 March 2023}

\begin{document}

\maketitle

\begin{abstract}
We consider an $n$-variate monomial function that is restricted both
in value by lower and upper bounds and in domain by two homogeneous
linear inequalities. Such functions are building blocks of several
problems found in practical applications, and that fall under the
class of Mixed Integer Nonlinear Optimization.  We show that the upper
envelope of the function in the given domain, for $n\ge 2$ is given by
a conic inequality. We also present the lower envelope for $n=2$. To
assess the applicability of branching rules based on homogeneous
linear inequalities, we also derive the volume of the convex hull for
$n=2$.
\end{abstract}

\section{Introduction}
\label{sec:intro}

Consider the function $f:\Rnp\to \mathbb R_+$ defined as $f(\x) =
\prod_{i\in N} x_i^{a_i}$ where $N = \{1,2,\ldots{},n\}$, $n>1$, $a_i
>0 \forall i\in N$. We are interested in the convex envelope of $f$ on
subsets of $\Rnp$ where the value of $f$ itself is bounded in the
range $[\ell,u]$, with $0 < \ell < u < +\infty$. Specifically, for a
given $\domn\subseteq \mathbb R_+^n$ we seek the convex hull of the set
$F(\domn) = \{(\x,z)\in (X \cap \domn) \times \mathbb R: z = f(\x)\}$,
where
\[
  X = \{
    \x \in \Rnp: \ell \le f(\x) \le u
  \}.
\]
Finding or approximating the convex hull of $F(\domn)$ is useful in
solving optimization problems whose objective function or constraints
contain polynomials with arbitrary real exponents. In particular,
those monomials with variables restricted to the first orthant are of
interest in the optimization with {\em posynomial} functions, or, more
in general, in geometric programming
\cite{boyd2007tutorial,ecker1980geometric}. Some algorithms for
posynomial optimization and for providing relaxations and lower bounds
for monomials on positive variables have been proposed
\citep{lu2010convex,tsai2011efficient}.

Monomials also appear in optimization problems with general nonlinear
constraints and discrete variables. Mixed-Integer Nonlinear
Optimization (MINLO) solvers such as Antigone
\citep{misener2014antigone}, Baron \citep{sah1996}, Couen\-ne
\citep{belotticouman09}, SCIP \citep{bestuzheva2021scip}, and
FICO-Xpress \citep{ficoxpress} employ a
{\em reformulation} scheme where expressions in factorable
programs are decomposed to smaller expressions that can be
targeted by an operator-specific convexification algorithm
\citep{s_and_costas,tawarmalani.sahinidis:02}. This allows for
exploiting branch-and-bound algorithms to compute the global
optimum of a MINLO problem \citep{horst.tuy:93}.
For instance, the generic polynomial constraint
\begin{equation}
  \label{eq:poly}
  \sum_{j\in J} c_j \prod_{i\in I_j} x_i^{a_{ji}} \le c_0,
\end{equation}
where $c_j\in \mathbb R$, $a_{ji} \in \mathbb R$, and $c_0\in \mathbb
R$, is in general nonconvex, and is decomposed by introducing {\em
  auxiliary} variables $y_j$ and $t_{ji}$ as follows:
\begin{eqnarray}
  \label{reform:lin}  &\sum_{j\in J} c_j y_j \le c_0\\
  \label{reform:prod} &y_j = \prod_{i\in I_j} t_{ji} & \forall j\in J\\
  \label{reform:pow}  &t_{ji} = x_i^{a_{ji}}     & \forall i\in I_j, \forall j\in J.
\end{eqnarray}
Exact solvers like those mentioned above find or approximate the
convex hull of the sets defined by the constraints associated with each
of the auxiliary variables. A convex relaxation of the feasible set
for \eqref{reform:lin}-\eqref{reform:pow}, which yields a valid lower
bound for the MINLO problem, is
\[
\textstyle R = R' \cap \left(\bigcap_{j\in J} R''_j \right) \cap \left(\bigcap_{j\in J, i\in I_j} R'''_{ji}\right),
\]
where
\[
  \begin{array}{lll}
  R'       & =         &                    \{(\x, \w, \t): \sum_{j\in J} c_j y_j \le b \},\\
  R''_j    & \supseteq & \textrm{conv}\left(\left\{(\x, \w, \t): y_j = \prod_{i\in I_j} t_{ji} \right\}\right),\\[2mm]
  R'''_{ji} & \supseteq & \textrm{conv}\left(\left\{(\x, \w, \t): t_{ji} = x_i^{a_{ji}}\right\}\right).
  \end{array}
\]
Clearly $R'$ is defined by a linear constraint and is added to the
reformulation as is. The reformulation results in a lower bound
on the optimal objective value that is tighter when $R''_j$ and $R'''_{ji}$ are tight
approximations of the corresponding convex hulls.
Note that even in the case $R''$ and $R'''$ are equal to, rather than
supersets of, the corresponding convex hulls, $R$ is still in general not the
convex hull of the set described by  \eqref{reform:lin}-\eqref{reform:pow}.
The tightness of
such relaxations strongly depends on tight bounds on both the original
variables $x_i$ and the auxiliaries $y_j$ and $t_{ji}$. Several {\em
  bound reduction} techniques such as feasibility-based (see
e.g.\mbox{} \citet{neu1990}) and optimality-based
\citep{tawarmalani.sahinidis:02} help find tight bounds on all
auxiliary variables.  Before generating an LP relaxation for
\eqref{eq:poly}, most solvers apply bound reduction to find a tight
bound interval $[\ell_j,u_j]$ on $y_j$ and $[\ell_{ji},u_{ji}]$ on
$t_{ji}$. This motivates the search for convex envelopes of $f$ over
$X \cap \domn$.

The equality signs in \eqref{reform:prod} and \eqref{reform:pow} can
be relaxed depending on the sign of $c_j$, which may  render
\eqref{reform:pow} convex depending on $a_{ji}$: if $a_{ji}\ge
1$, then the constraint $t_{ji} \ge x_i^{a_{ji}}$ is convex, and
viceversa if $a_{ji} \le 1$ then $t_{ji} \le x_i^{a_{ji}}$ is
convex.

However, \eqref{reform:prod} is nonconvex regardless of the
sign, although it admits a polyhedral convex hull
\citep{rik1997}. \citet{mccormick76} provides a valid relaxation for
the product of two variables,
\begin{equation}
B = \{(x_1, x_2, z)\in \mathbb R^3_+: z = x_1 x_2, (x_1, x_2) \in [\ell_1, u_1] \times [\ell_2, u_2]\},
\end{equation}
which consists of the four so-called {\em McCormick inequalities}:
\begin{eqnarray}
  \label{eq:mccormick1} z & \ge & \ell_2 x_1 + \ell_1 x_2 - \ell_2 \ell_1\\
                        z & \ge &    u_2 x_1 +    u_1 x_2 -    u_2    u_1\\
                        z & \le & \ell_2 x_1 +    u_1 x_2 - \ell_2    u_1\\
  \label{eq:mccormick4} z & \le &    u_2 x_1 + \ell_1 x_2 -    u_2 \ell_1.
\end{eqnarray}
These inequalities, in fact, form the convex hull of $B$
\citep{alkhayyal.falk:83}. While these results hold without sign
constraints on $x_1$ and $x_2$, we use the above definition of $B$ as
we focus on sets entirely contained in $\Rnp$.
\citet{meyerfl2004} present inequalities of the
convex hull in the trilinear case ($n=3$).

\subsection{The case for bounded monomials}

Bound reduction that uses the bounds on the initial variables $\x$ and
the linear constraint \eqref{reform:lin} can obtain, as an
implication, lower and upper bound on each element of the sum in
\eqref{eq:poly}. In the bilinear case, i.e., $n=2$, $a_1 = a_2 = 1$,
the convex hull of
\[
B' = \{(x_1, x_2, z) \in B: \ell\le z \le u\}
\]
is tighter than
\eqref{eq:mccormick1}-\eqref{eq:mccormick4} if the bounds on $z$ are
tighter than those on $x_1$ and $x_2$, i.e., $\ell > \ell_1 \ell_2$ or $u
< u_1 u_2$.

\citet{belotti2010805} introduce inequalities that are not implied by
the McCormick inequalities if at least one of the bounds on $z$ is tighter,
and that result in a tighter relaxation
when solving quadratically constrained MINLO problems.
\citet{anstreicher2021convex} show that the convex hull of $B'$ for
``tight'' values of $\ell$ or $u$ is the union of distinct regions, each
partially delimited by a different second-order cone,
i.e., a set of the form $\bb^\top \x + b_0 \ge \|G\x + \gg\|_2$.
\citet{nguyen2018deriving} obtain envelopes of monomials
$x_1^{a_1} x_2^{a_2}$ for the cases $a_1=1\le a_2$ (convex hull and
upper envelope) and $a_1, a_2 \ge 1$ (lower envelope).

This article mainly focuses on the convex hull of $F(\domn)$ where
\begin{equation}
  \label{eq:def_wij}
  \domn = W_{ij} := \{\x \in \Rnp: p x_i \le x_j \le q x_i\},
\end{equation}
with $0 < p < q$, for two indices $i,j\in N$. Therefore we depart from
the most usual setting where lower and upper bounds on $\x$ are given,
and instead constrain the $\x$ variables with two homogeneous
(i.e.\mbox{} zero right-hand side) linear inequalities, which restrict
the $\x$ space to a {\em wedge}.

Using linear inequalities instead of a bounding box as the feasible
set for $\x$ deviates from the common setting exploited by
MINLO solvers.
An advantage of this approach is that we may obtain a tighter
relaxation by avoiding the decomposition of the factorable term
$\prod_{i\in N} x_i^{a_i}$ into $y=\prod_{i\in N} t_i$ and $t_i =
x_i^{a_i}$.

While we assume $a_i>0\forall i\in N$, note that we can reduce
monomials with one or more terms $x_i^{a_i}$ with $a_i<0$ to this case
by another reformulation step: introduce an auxiliary variable $y_i$
such that $y_i = \frac{1}{x_i}$ and replace $x_i^{a_i}$ with
$y_i^{-a_i}$. This requires $x_i>0 \forall \x\in X \cap \domn$, which
is implied by $\ell > 0$, and it introduces an extra gap due to the
additional reformulation step $\{(x_i,y_i) \in \mathbb R_+^2: x_i y_i
= 1\}$, whose convex hull is $\{(x_i,y_i) \in \mathbb R_+^2: x_i y_i
\ge 1\}$ if $x_i \in [0,+\infty]$, otherwise it is $\{(x_i,y_i) \in
\mathbb R_+^2: x_i y_i \ge 1, x_i + u \ell y_i \le \ell + u\}$.

\begin{defn}
  Given $T\subseteq \mathbb R^n$ and a function $f: T \to \mathbb R$,
  the {\em epigraph} of $f$ in $T$ is $\epig(f,T) = \{(\x,z)\in
  T\times \mathbb R: z \ge f(\x)\}$. The {\em hypograph} of $f$ in $T$
  is $\hypg(f,T) = \{(\x,z)\in T\times \mathbb R: z \le
  f(\x)\}$.
\end{defn}

\begin{defn}
  Given $T\subseteq \mathbb R^n$ and a function $f: T\to \mathbb R$,
  the {\em lower envelope} $E_L(f,T)$ (resp. {\em upper envelope}
  $E_U(f,T)$) of $f$ over $T$ is the convex hull of the epigraph
  (resp.\mbox{} hypograph) of $f$ in $T$.
\end{defn}

In Section \ref{sec:hulls} we provide some trivial results on the
convex hull of $F(\Rnp)$. In Section \ref{sec:upper_env} we present
the upper envelope of $f$ over $X \cap W_{ij}$ for any $n\ge 2$. In
Section \ref{sec:lower_env} we present the lower envelope of $f$ over
$X \cap W_{ij}$ for $n=2$, yielding the convex hull of $F(W_{12})$.
In Section \ref{sec:volume} we compute the volume of the convex hull
of $F(W_{12})$ for $n=2$; concluding remarks are in Section
\ref{sec:conclusions}.

\section{Convex hull of $F(\Rnp)$}
\label{sec:hulls}

Define $\beta = \sum_{i\in N} a_i$, then for $z_0,\gamma\in
\mathbb R$ consider the cone
\[
\textstyle \mathcal K = \left\{(\x,z) \in \Rnp \times \mathbb R:
  (z - z_0)^\beta \le \gamma \prod_{i\in N} x_i^{a_i}\right\}.
\]
The vertex of $\mathcal K$ is $({\mathbf 0}, z_0)$.  Also define
$F(\domn)^\le := \{(\x,z) \in (X\cap \domn) \times \mathbb R: z \le
f(\x)\}$; note that it is obviously not the hypograph of $f$ but
rather a relaxation of the link between $z$ and $f$. If $\beta = 1$
then $F(\domn)^\le$ is a convex cone intersected with $S :=
\{(\x,z)\in \Rnp \times \mathbb R: \ell\le z \le u\} = \Rnp \times
  [\ell,u]$. Similar to the bilinear case, where $n=2$ and $a_1 = a_2 =
  1$, we look for conditions under which $\mathcal K$ defines a tight
  relaxation of $F(\domn)$.
Let us find $z_0$ and $\gamma$ such that
\[
  \begin{array}{ll}
    \{(\x,z) \in F(\domn)^\le: z = \ell\} = \{(\x,z) \in \mathcal K: z = \ell\};\\
    \{(\x,z) \in F(\domn)^\le: z = u   \} = \{(\x,z) \in \mathcal K: z = u\}.
  \end{array}
\]
This is equivalent to finding $z_0, \gamma$ such that
\begin{equation}
  \label{eq:equalintersections}
  \begin{array}{ll}
  \ell \le \prod_{i\in N} x_i^{a_i} \Leftrightarrow \frac{1}{\gamma}(\ell-z_0)^\beta \le \prod_{i\in N} x_i^{a_i} & \forall \x \in \Rnp;\\[1mm]
  u \le \prod_{i\in N} x_i^{a_i} \Leftrightarrow \frac{1}{\gamma}(u-z_0)^\beta \le \prod_{i\in N} x_i^{a_i} & \forall \x \in \Rnp\\
  \end{array}
\end{equation}
and hence
$
    \ell = \frac{1}{\gamma}(\ell - z_0)^\beta,
    u = \frac{1}{\gamma}(u - z_0)^\beta
$,
which implies $z_0 \le \ell$.  Because $u>0$ and $u-z_0 > 0$, we
divide the first equation by the second one and obtain $
\left(\frac{\ell}{u}\right)^{1/\beta}= \frac{(\ell-z_0)}{(u - z_0)} $.
This yields
\begin{equation}
  \label{eq:defzgamma}
  z_0 = \frac{u^{\frac{1}{\beta}} \ell - \ell^{\frac{1}{\beta}} u}{u^{\frac{1}{\beta}} - \ell^{\frac{1}{\beta}}}, \quad \gamma = \left(\frac{u-\ell}{u^{\frac{1}{\beta}} - \ell^{\frac{1}{\beta}}}\right)^\beta.
\end{equation}
Note that $ \beta \ge 1 \Leftrightarrow (z_0 \le 0, \gamma \ge 1)$.
Also, if $\beta = 1$, $z_0 = 0$,
$\gamma = 1$, and it is easy to verify that $\mathcal K \cap S \equiv
F(\Rnp)^\le$, while $\ell=0$ implies $z_0 = 0$, i.e., the vertex of
$\mathcal K$ is the origin.


\begin{lemma}
  \label{lemma:convhyperb}
  For $a_i > 0\forall i\in N$, $\ell \ge 0$, $T = \{\x \in \Rnp: \prod_{i\in
    N} x_i^{a_i} \ge \ell\}$ is convex.
\end{lemma}

\begin{proof}
  If $\ell = 0$, $T = \Rnp$. Assume now $\ell > 0$; then $x_i > 0,
  i\in N$, for all $\x\in T$. Consider $\x', \x'' \in T$. For any $\mu
  \in [0,1]$, we prove $\prod_{i\in N} (\mu x'_i + (1-\mu)
  x''_i)^{a_i} \ge \ell$, i.e.,
  \[
  \sum_{i\in N} a_i\log\left(\mu x'_i + (1-\mu) x''_i\right) \ge \log \ell.
  \]
  Concavity of the $\log$ function and positivity of all $a_i$'s imply
  \[
  \begin{array}{ll}
       & \sum_{i\in N} a_i\log\left(\mu x'_i + (1-\mu) x''_i\right) \ge \\
   \ge & \sum_{i\in N} a_i(\mu \log x'_i + (1-\mu) \log x''_i) = \\
   =   & \mu \sum_{i\in N} a_i \log x'_i + (1-\mu) \sum_{i\in N} a_i \log x''_i \ge \\
   \ge & \mu \log \ell + (1-\mu)  \log \ell = \log \ell,
  \end{array}
  \]
  which proves convexity of $T$.
\end{proof}

\begin{lemma}
  \label{lemma:concratio}
  Let $g:\mathbb R\to \mathbb R$ be a convex (resp.\mbox{} concave)
  monotone non-decreasing function. Then
  for any $\ell, z, u$ such that $\ell < z < u$,
  \[
  \frac{g(z) - g(\ell)}{z - \ell} \le \frac{g(u) - g(z)}{u - z} \qquad
  \left(\textrm{resp. }
  \frac{g(z) - g(\ell)}{z - \ell} \ge \frac{g(u) - g(z)}{u - z}\right).
  \]
\end{lemma}

\begin{proof}
  Consider $\mu\in(0,1)$ such that $z = (1 - \mu) \ell + \mu u$. Then
  $0 < \mu < 1$ and for $g$ concave we obtain $g(z) \ge (1 - \mu)
  g(\ell) + \mu g(u)$, or
  \[
    \begin{array}{rcl}
      (1 - \mu) g(z) + \mu g(z)  & \ge &  (1 - \mu) g(\ell) + \mu g(u)\\
      (1 - \mu)   z  + \mu   z   &  =  &  (1 - \mu)   \ell  + \mu   u \\[2mm]
      (1 - \mu) (g(z) - g(\ell)) & \ge &   \mu (g(u) - g(z))          \\
      (1 - \mu) (z - \ell)       &  =  &   \mu (u - z).               \\
    \end{array}
  \]
  Dividing the third inequality by the last equation yields
  $\frac{g(z) - g(\ell)}{z - \ell} \ge \frac{g(u) - g(z)}{u -
    z}$. Similar considerations hold, {\em mutatis mutandis}, for $g$ convex.
\end{proof}

\begin{lemma}
  \label{lemma:containment}
  $F(\Rnp)\subseteq \mathcal K$ if and only if $\beta \ge 1$.
\end{lemma}

\begin{proof}
  $F(\Rnp)\subseteq \mathcal K$ if and only if $(\ell\le z\le u) \wedge (z =
  \prod_{i\in N} x_i^{a_i}) \Rightarrow (z-z_0)^\beta \le \gamma
  \prod_{i\in N} x_i^{a_i}$. Rewrite the left-hand side as $ (z-
  z_0)^\beta\le \gamma z$. Replacing $z_0$ and $\gamma$ we get
  \[
  \left(z-\frac{u^{1/\beta} \ell - \ell^{1/\beta} u}{u^{1/\beta} - \ell^{1/\beta}}\right)^\beta \le \left(\frac{u-\ell}{u^{1/\beta} - \ell^{1/\beta}}\right)^\beta z,
  \]
  from which we eliminate the (positive) common denominator:
  \[
  \begin{array}{lcll}
      & \textstyle \left(z u^{\frac{1}{\beta}} - z \ell^{\frac{1}{\beta}} -u^{\frac{1}{\beta}}\ell + \ell^{\frac{1}{\beta}}u\right)^\beta &= \\
    = & \textstyle\left(u^{\frac{1}{\beta}}(z-\ell)+\ell^{\frac{1}{\beta}}(u-z)\right)^\beta &\le& \textstyle(u-\ell)^\beta z.
  \end{array}
  \]
  Because $\beta > 0$ and all terms in both left- and right-hand side
  are nonnegative, the above is equivalent to
  \begin{equation}
    \label{eq:almostconvcomb}
    u^{\frac{1}{\beta}}(z-\ell)+\ell^{\frac{1}{\beta}}(u-z)\le (u-\ell)z^{\frac{1}{\beta}}.
  \end{equation}
Rewrite the right-hand side as $(u-z)z^{\frac{1}{\beta}} +
(z-\ell)z^{\frac{1}{\beta}}$ and \eqref{eq:almostconvcomb} as
$(u-z)(z^{\frac{1}{\beta}} - \ell^{\frac{1}{\beta}}) \ge
(z-\ell)(u^{\frac{1}{\beta}} - z^{\frac{1}{\beta}})$, or
\[
  \frac{z^{\frac{1}{\beta}} - \ell^{\frac{1}{\beta}}}{z-\ell} \ge
  \frac{u^{\frac{1}{\beta}} - z^{\frac{1}{\beta}}}{u-z},
\]
which, for Lemma \ref{lemma:concratio}, is true if and only if $\beta
\ge 1$ for concavity of $g(x) = x^{\frac{1}{\beta}}$.
\end{proof}

The structure of both
upper envelope and lower envelope of $f$ over $X \cap \domn$ discussed
in this and the following section changes radically at $\beta=1$. For
instance, $F(\Rnp)^\le$ is convex for $\beta \le 1$ and nonconvex for
$\beta > 1$. Therefore we split the treatment within each section
depending on the ranges of $\beta$.

\begin{prop}
  \label{prop:upper}
  If $\beta \ge 1$, then $\textrm{conv}(F(\Rnp)) = \mathcal K \cap S$.
\end{prop}

\begin{proof}
  $F(\Rnp) \subseteq \mathcal K$ from Lemma \ref{lemma:containment};
  this implies that $\textrm{conv}(F(\Rnp)) \subseteq \mathcal K$ as
  $\mathcal K$ is convex. In addition, $F(\Rnp)\subseteq S$ and
  convexity of $S$ imply $\textrm{conv}(F(\Rnp)) \subseteq \mathcal K
  \cap S$. We just need to prove $\textrm{conv}(F(\Rnp))\supseteq
  \mathcal K \cap S$.

  Consider $(\tilde \x, \tilde z)\in \mathcal K\cap S$, i.e., $(\tilde
  z-z_0)^\beta \le \gamma \prod_{i\in N} \tilde x_i^{a_i}$, $\ell\le
  \tilde z \le u$. Let us construct $\x'$ and $\x''$ such that
  $(\tilde \x, \tilde z - z_0)$ is a convex combination of $(\x', \ell
  - z_0)$ and $(\x'', u - z_0)$.  Let $(\x', \ell - z_0) = s' (\tilde
  \x, \tilde z - z_0)$ and $(\x'',u - z_0) = s''(\tilde \x, \tilde z -
  z_0)$, i.e., $(\x', \ell)$ and $(\x'', u)$ lie on the same ray of
  $\mathcal K$ as $(\tilde \x, \tilde z)$. By construction of $s'$ and
  $s''$,
  \[
    s'  = \frac{\ell - z_0}{\tilde z - z_0}\le 1; \qquad
    s'' = \frac{u    - z_0}{\tilde z - z_0}\ge 1.
  \]
  In addition, $(\x',\ell)\in \mathcal K$ and $(\x'',u)\in
  \mathcal K$ because $(\tilde \x, \tilde z)\in \mathcal K$.
  Specifically, they belong to $\{(\x,z) \in \mathcal K:
  z = \ell\}$ and $\{(\x,z) \in \mathcal K: z = u\}$, respectively,
  but by \eqref{eq:equalintersections} this implies
 $(\x',\ell)\in F^\le_\ell := \{(\x,z) \in F(\Rnp)^\le: z = \ell\}$
 and
 $ (\x'',u)\in F^\le_u := \{(\x,z) \in F(\Rnp)^\le: z = u\}$.
 Because both $F^\le_\ell$ and $F^\le_u$ are convex by Lemma
 \ref{lemma:convhyperb}, there exist two extreme points $\x'_a$ and
 $\x'_b$ of $F(\Rnp)$ of which $\x'$ is a convex combination, such
 that $f(\x'_a)=f(\x'_b) = \ell$, and similar for $\x''$, thus
 implying $(\tilde \x, \tilde z)$ is a convex combination of
 points of $F(\Rnp)$.
\end{proof}

\begin{prop}
  If $\beta \le 1$, then $\conv(F(\Rnp)) = F(\Rnp)^\le$.
\end{prop}

\begin{proof}
  $F(\Rnp)\subseteq F(\Rnp)^\le$ and $F(\Rnp)^\le$ is convex, therefore
  $\textrm{conv}(F(\Rnp)) \subseteq F(\Rnp)^\le$. To prove
  $\textrm{conv}(F(\Rnp)) \supseteq F(\Rnp)^\le$, consider $(\tilde
  \x, \tilde z)\in F(\Rnp)^\le$. If $f(\tilde \x)=\tilde z$, obviously
  the result holds. If $\tilde z < f(\tilde \x)$, similar to the proof
  of Proposition \ref{prop:upper}, convexity of $T = \{(\x, z) \in
  F(\Rnp)^\le: z = \tilde z\}$ implies there exist two extreme points
  of $T$, i.e., two elements of $\{(\x,z)\in F(\Rnp): z=\tilde z\}$,
  of which $(\tilde \x, \tilde z)$ is a convex combination.
\end{proof}

\section{Upper envelope over $X \cap W_{ij}$}
\label{sec:upper_env}

We established that the convex hull of $F(\Rnp)$ is defined by
$\mathcal K \cap S$ if $\beta \ge 1$ and $F(\Rnp)^\le$ otherwise. From now
on, we consider $\domn=W_{ij}$ defined in \eqref{eq:def_wij}:
\[
  W_{ij} = \{\x \in \Rnp: p x_i \le x_j \le q x_i\}.
\]
Below we prove that the above result on the upper envelope is
substantially unchanged, save an extra inequality for $n=2$.
\begin{lemma}
  \label{lemma:n2finite}
  For any $0 < p \le q < +\infty$, the set $X \cap W_{ij}$ is bounded
  for $n=2$ and unbounded for $n>2$.
\end{lemma}
\begin{proof}
  For $n=2$, $x_1^{a_1}x_2^{a_2}\le u$ and $x_2 \ge p x_1$ imply
  $p^{a_2} x_1^{a_1 + a_2} \le x_1^{a_1}x_2^{a_2} \le u$, i.e., $x_1
  \le \omega_1 := (u p^{-a_2})^{\frac{1}{a_1+a_2}}$. Similarly,
  $x_1^{a_1}x_2^{a_2}\le u$ and $x_2 \le q x_1$ imply $q^{-a_1}
  x_2^{a_1 + a_2} \le x_1^{a_1}x_2^{a_2} \le u$, i.e., $x_2 \le
  \omega_2 := (u q^{a_1})^{\frac{1}{a_1+a_2}}$. Therefore $X \cap
  W_{ij}$ is contained in the bounding box $[0,\omega_1] \times
  [0,\omega_2] \times [\ell,u]$. Note that $x_1$ and $x_2$ may have
  tighter lower bounds than 0, but this is out of the scope of this
  proof.

  For $n>2$, choose $h\in N\setminus \{i,j\}$ and $r \in [p,q]$, then
  let $\x\in \Rnp$ be such that
  \begin{itemize}
    \item $x_h = M^{\frac{1}{a_h}}$ where $M$ is an
      arbitrarily large number;
    \item $x_i = \left(\frac{\ell}{M r^{a_j}}\right)^{\frac{1}{a_i + a_j}}$;
    \item $x_j = r x_i$;
    \item $x_k = 1\,\forall k\notin \{i,j,h\}$.
  \end{itemize}
  Then $\x\in X \cap W_{ij}$ since $\prod_{k\in N} x_k^{a_k} =
  x_h^{a_h} x_i^{a_i} x_j^{a_j} = \ell$ and $p x_i \le r x_i = x_j \le
  q x_i$, for arbitrarily large $M$.
\end{proof}

\begin{prop}
  \label{prop:upper_env}
  If $\beta \ge 1$ and $n>2$, the upper envelope of $f$ over $X \cap W_{ij}$ is
  \begin{eqnarray}
    \nonumber &H =\!\!& \{(\x,z)\in \Rnp\times \mathbb R: \\
    \label{eqnpf:upper_cone} &&\textstyle(z-z_0)^\beta \le \gamma\prod_{k\in N} x_k^{a_k}\\
    \label{eqnpf:zub}        && z\le u\\
    \label{eqnpf:wedge}      && p x_i \le x_j \le q x_i\\
    \label{eqnpf:prodlb}     &&\textstyle \prod_{k\in N} x_k^{a_k} \ge \ell\}.
  \end{eqnarray}
\end{prop}

\begin{proof}
  $H$ is
  convex: \eqref{eqnpf:upper_cone} is a convex cone, \eqref{eqnpf:zub}
  and \eqref{eqnpf:wedge} are linear, and \eqref{eqnpf:prodlb} defines
  a convex set for Lemma \ref{lemma:convhyperb}.
  Obviously $\hypg(f,X \cap W_{ij})\subseteq H$ since any
  $(\x,z)\in \hypg(f,X \cap W_{ij})$ satisfies
  \eqref{eqnpf:upper_cone} by Lemma \ref{lemma:containment} and
  \eqref{eqnpf:zub}-\eqref{eqnpf:prodlb} by construction. Therefore $E_U(f,X
  \cap W_{ij})\subseteq H$. We prove now that $E_U(f,X \cap
  W_{ij})\supseteq H$.

  Consider $(\tilde \x, \tilde z)\in H$. If $f(\tilde \x) \le u$ and
  $\tilde z \le f(\tilde \x)$, then $(\tilde \x, \tilde z)\in
  \hypg(f,X \cap W_{ij})$ and the result holds. Otherwise,
  there are two cases: $\tilde z > f(\tilde \x)$ and $f(\tilde \x) >
  u$, which are mutually exclusive since $\tilde z > f(\tilde \x) > u$
  violates \eqref{eqnpf:zub}.

  \subparagraph*{Case 1:  $\tilde z > f(\tilde \x)$.}
  Define $f'_u(\x) = z_0 + \left(\gamma\prod_{k\in N}
  x_k^{a_k}\right)^{\frac{1}{\beta}}$.  Since $f(\tilde \x) \le u$, we can
  construct two points $\x'=s'\tilde \x$ and $\x''=s''\tilde \x$, with
  $s'\le 1 \le s''$, such that $f(\x')=\ell$ and $f(\x'')=u$. Because
  $f(s\x)=s^{\beta}f(\x)\forall s\ge 0$, both $s'$ and $s''$
  exist. Moreover, note that $f'_u(s\x)=sf'_u(\x)\forall s\ge 0$ as
  $f'_u$ is a conic function.  By construction of $f'_u$ in Section
  \ref{sec:hulls}, $f'_u(\x') = f(\x') = \ell$ and $f'_u(\x'') =
  f(\x'') = u$. Then $(\tilde \x, f'_u(\tilde \x))\in E_U(f,X \cap
  W_{ij})$ as it is a convex combination of $(\x',\ell)$ and $(\x'',
  u)$, and since $f(\tilde \x) < \tilde z \le f'_u(\tilde \x)$,
  $(\tilde \x, \tilde z)$ is a convex combination of $(\tilde \x,
  f'_u(\tilde \x))$ and $(\tilde \x, f(\tilde \x))\in F(W_{ij})$, thus
  proving that $(\tilde \x, \tilde z)\in E_U(f,X \cap W_{ij})$. Note
  that this part of the proof is also valid for $n = 2$.

  \subparagraph*{Case 2: $f(\tilde \x) > u$.}
  We find $\x'$, $\x''\in F(W_{ij})$ such
  that $\tilde \x$ is a convex combination of $\x'$ and $\x''$. Define
  $r = \tilde x_j / \tilde x_i > 0$; note that $p \le r \le q$. Since
  $n>2$, select $h\in N \setminus \{i,j\}$, then define the parametric
  point $\x(t)$ as follows:
  \[
    \begin{array}{ll}
      x_i(t) = \tilde x_i - t\\
      x_j(t) = r x_i = \tilde x_j - r t = r(\tilde x_i - t)\\
      x_h(t) = \tilde x_h + t\\
      x_k(t) = \tilde x_k \, \forall k\in N \setminus \{i,j,h\}.
    \end{array}
  \]
  Then $g(t) := f(\x(t)) = r^{a_j}(\tilde x_i - t)^{a_i + a_j} (\tilde
  x_h + t)^{a_h} \prod_{k\notin\{i,j,h\}} \tilde x_k$ is a continuous
  function such that $f(\x(0)) > u>0$ and $f(\x(t)) =0$ for $t=-\tilde
  x_h<0$ and for $t=\tilde x_i>0$. Hence there are two values $t',t''$
  such that $-\tilde x_h < t' < 0 < t'' < \tilde x_i$ and
  $f(\x(t'))=f(\x(t''))=u$. Because both $(\x(t'),u)$ and
  $(\x(t''),u)$ are in $F(W_{ij})$ and $f(\tilde \x) \le \tilde z$,
  both $(\x(t'),\tilde z)$ and $(\x(t''),\tilde z)$ are in $\hypg(f,X
  \cap W_{ij})$, implying that $(\tilde \x,\tilde z)$ is a convex
  combination of two points of $\hypg(f,X \cap W_{ij})$. This argument
  also holds for $\beta \le 1$.
\end{proof}

\begin{prop}
  If $\beta \le 1$ and $n>2$, the upper envelope of $f$ over $X \cap W_{ij}$ is
  \begin{eqnarray}
    \nonumber &H =\!\! & \{(\x,z)\in \Rnp\times \mathbb R: \\
    \label{eqnpf2:upper_cone} && \textstyle z \le \prod_{k\in N} x_k^{a_k}\\
    \label{eqnpf2:zub}        && z\le u\\
    \label{eqnpf2:wedge}      && p x_i \le x_j \le q x_i\\
    \label{eqnpf2:prodlb}     && \textstyle \prod_{k\in N} x_k^{a_k} \ge \ell\}.
  \end{eqnarray}
\end{prop}

\begin{proof}
  First, $\hypg(f,X \cap W_{ij}) \subseteq H$ because any
  $(\x,z)\in \hypg(f,X \cap W_{ij})$ satisfies
  \eqref{eqnpf2:upper_cone}-\eqref{eqnpf2:prodlb}.  $H$ is also
  convex: \eqref{eqnpf2:upper_cone} is convex because $f(\x)$ is
  concave for $\beta \le 1$ and \eqref{eqnpf2:prodlb} is convex for
  Lemma \ref{lemma:convhyperb}. Therefore we also have $E_U(f,X \cap
  W_{ij}) \subseteq H$. We now prove $H \subseteq E_U(f,X \cap
  W_{ij})$.

  Consider $(\tilde \x, \tilde z) \in H$. If $\tilde z = f(\tilde
  \x)$, then obviously $(\tilde \x, \tilde z) \in X \cap W_{ij}$ and
  the result holds. Let $\tilde z < f(\tilde \x)$. If $f(\tilde \x)
  \le u$, then $\left(\tilde \x, f(\tilde \x)\right)\in F(W_{ij})$ and
  $(\tilde \x, \tilde z)$ belongs to its hypograph and hence to its
  upper envelope. If $f(\tilde \x) > u$, then similar to
  Case 2 of Proposition \ref{prop:upper_env}, one can obtain two
  vectors in $\hypg(f,X \cap W_{ij})$ of which $(\tilde \x, \tilde z)$
  is a convex combination.
\end{proof}

\section{Lower envelope over $X \cap W_{ij}$}
\label{sec:lower_env}

We build on a property of the monomial function $f$ for general $n\ge
2$ to derive a few results leading to the lower envelope of $f$ over
$X \cap W_{ij}$ for $n=2$. Define the level set $C_\xi = \{\x\in \Rnp:
f(\x) = \xi\}$ and the two subspaces $P_{ij} = \{\x\in \Rnp: x_j = p
x_i\}$ and $Q_{ij} = \{\x\in \Rnp: x_j = q x_i\}$. There
is a bijection from $C_\xi\cap P_{ij}$ to $C_\xi \cap Q_{ij}$ such
that all pairs of points are joined by parallel lines.

\begin{figure}[t]

  \centering

  \def\pa{0.55} 
  \def\qa{2.4}
  \def\expAI{1}
  \def\expAJ{1}


  \subfigure[$\xi\in\{2,4,8,16,32\}$, $(a_1, a_2) = (\expAI,\expAJ)$,
      $(p,q)=(\pa,\qa)$]{
    \includegraphics[width=.45\textwidth]{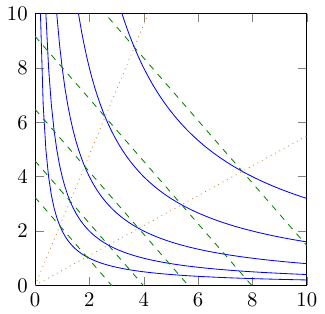}
  }
%
%
  \def\pb{0.19}
  \def\qb{3.6}
  \def\expBI{0.3}
  \def\expBJ{0.6}
\subfigure[$\xi\in\{1,2,3\}$, $(a_1, a_2) =
    (\expBI,\expBJ)$, $(p,q)=(\pb,\qb)$]{
  \includegraphics[width=.45\textwidth]{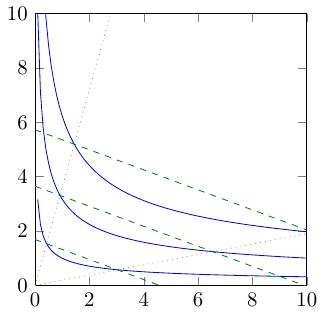}
  }

  \def\pc{1.5}
  \def\qc{4.9}
  \def\expCI{3.9}
  \def\expCJ{7.6}

  \subfigure[$\xi\in\{8,2^9,2^{15},2^{21}\}$, $(a_1, a_2) = (\expCI,\expCJ)$,
      $(p,q)=(\pc,\qc)$]{
    \includegraphics[width=.45\textwidth]{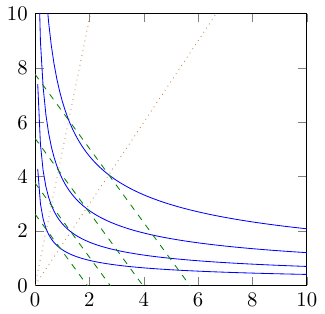}
  }
%
%
  \def\pd{0.2}
  \def\qd{9}
  \def\expDI{0.6}
  \def\expDJ{0.4}
  \subfigure[$\xi\in\{0.5,1,2\}$, $(a_1, a_2) = (\expDI,\expDJ)$,
      $(p,q)=(\pd,\qd)$]{
    \includegraphics[width=.45\textwidth]{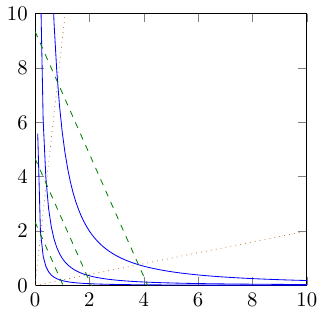}
  }
  \caption{Solid lines are the level curves $\{\x\in \mathbb R_+^2:
    x_1^{a_1}x_2^{a_2} = \xi\}$; the dotted lines $x_2 = p x$ and $x_2
    = q x_1$ intersect each level curve in one point each. The dashed
    lines through the intersections of the level curves and the dotted
    lines are parallel to one another.}
  \label{fig:parallels}
\end{figure}

Let us illustrate the property on
bivariate functions first. It is easy to show that for $n=2$ and $a_1
= a_2 = 1$, $C_\xi \cap P_{ij} = \{\check{\x}\}$ and $C_\xi \cap Q_{ij} =
\{\hat{\x}\}$ where
\[
  \check{\x} = \left(\sqrt{p\xi}, \sqrt{\frac{1}{p}\xi}\right), \quad
  \hat{\x} = \left(\sqrt{q\xi}, \sqrt{\frac{1}{q}\xi}\right).
\]
The line through $\check{\x}$ and $\hat{\x}$ is given by the linear equation
\[
\det\left(\left[
\begin{array}{ccc}
x_1           & x_2                     & 1\\
\sqrt{p\xi} & \sqrt{\frac{1}{p}\xi} & 1\\
\sqrt{q\xi} & \sqrt{\frac{1}{q}\xi} & 1\\
\end{array}
\right]\right) = 0,
\]
and its slope $\frac{1/\sqrt{p} - 1/\sqrt{q}}{\sqrt{p} - \sqrt{q}}
= -(pq)^{-\frac{1}{2}}$ is independent of $\xi$. This also
holds for general positive exponents $(a_1, a_2)\neq(1,1)$, in which
case the coefficients of $x_1$ and $x_2$ are proportional to
$\xi^{\frac{1}{\beta}}$ rather than $\sqrt{\xi}$.
Figure \ref{fig:parallels} shows examples for different values of
$a_1$, $a_2$, $p$, $q$, and $\xi$.
We prove that this
generalizes to $n\ge 2$.

\begin{lemma}
  Given $\aa \in \Rnp$, $i,j\in N$, $\xi \in \mathbb R$, $p, q \in \mathbb R_+$ with $i\neq
  j$ and $0 < p < q$, there exist $d_i < 0$ and $d_j > 0$ such that
  for any $\check \x$ that satisfies $\check x_j = p \check x_i$ and
  $\prod_{i\in N} \check x_i^{a_i} = \xi$, there exists a unique
  solution $(\bar s,\bar \x)\in \mathbb R_+ \times \Rnp$ to the
  nonlinear system
  \begin{eqnarray}
    \label{eq:qxzero}  &&\bar x_j = q \bar x_i\\
    \label{eq:barxij}  &&(\bar x_i, \bar x_j) = (\check x_i + \bar s d_i, \check x_j + \bar s d_j)\\
    \label{eq:barx}    &&\bar x_k = \check x_k \quad \forall k\notin\{i,j\}\\
    \label{eq:barprod} &&\textstyle\prod_{i\in N} \bar x_i^{a_i} = \xi.
  \end{eqnarray}
\end{lemma}

\begin{proof}
We prove that there exist $\eta_1$ and $\eta_2$ such that $(\bar
x_i,\bar x_j) = (\eta_i \check x_i, \eta_j \check x_j)$.
First, \eqref{eq:barprod} requires $\prod_{k\in N} \bar x_k^{a_k} =
\eta_i^{a_i} \eta_j^{a_j} \prod_{k\in N} \check x_k^{a_k} = \xi$,
i.e., $\eta_i^{a_i} \eta_j^{a_j} = 1$.
Given that $\check x_j = p \check x_i$ and $\eta_j \check x_j= q
\eta_i \check x_i$, solving both for $\frac{\check x_i}{\check x_j}$
yields $\frac{1}{p} = \frac{\eta_j}{q \eta_i}$, or
\begin{equation}
  \label{eq:pqratio}
  \frac{\eta_j}{\eta_i} = \frac{q}{p}.
\end{equation}
This yields $\eta_i = \left(\frac{q}{p}\right)^{-\frac{a_j}{a_i +
    a_j}} < 1$ and $\eta_j=\eta_i \frac{q}{p} =
\left(\frac{q}{p}\right)^{\frac{a_i}{a_i + a_j}} > 1$.

In order to obtain $d_i$ and $d_j$, observe that
\[
  \begin{array}{lll}
    \check x_i + \bar s d_i = \eta_i \check x_i\\
    \check x_j + \bar s d_j = \eta_j \check x_j.\\
  \end{array}
\]
Solving both for $\bar s$, we have
$
\frac{1}{d_i} \check x_i(\eta_i - 1) =
\frac{1}{d_j} \check x_j(\eta_j - 1)
$,
and since $\frac{\check x_i}{\check x_j} = \frac{1}{p}$ we get
\begin{equation}
  \label{eq:didjratio}
  \frac{d_i}{d_j}\frac{\eta_j - 1}{\eta_i - 1} = \frac{1}{p}.
\end{equation}
%
For
convenience, define fractions $\varphi_i = \frac{a_i}{a_i+a_j}$ and $\varphi_j =
\frac{a_j}{a_i+a_j}$. Note that
\[
  \frac{\eta_j-1}{\eta_i-1} =
  \frac{q^{\varphi_i}- p^{\varphi_i}}{p^{\varphi_i}} \cdot
  \frac{p^{-\varphi_j}}{q^{-\varphi_j}- p^{-\varphi_j}} =
  \frac{q^{\varphi_i}- p^{\varphi_i}}{q^{-\varphi_j}- p^{-\varphi_j}}\cdot\frac{1}{p}.
\]
Because $d_i,d_j$ define a direction, rather than normalizing them
(which would lead to complications in the next section) we determine
$d_i, d_j$ as denominator and numerator of the fraction above:
\[
  \begin{array}{llclc}
    d_i &=& q^{-\varphi_j} - p^{-\varphi_j} &=& q^{-a_j/(a_i+a_j)} - p^{-a_j/(a_i+a_j)}\\
    d_j &=& q^{\varphi_i}  - p^{\varphi_i}  &=& q^{a_i/(a_i+a_j)}  - p^{a_i/(a_i+a_j)}.
  \end{array}
\]
Note that (i) $d_i < 0 < d_j$  (ii) both $d_i$ and $d_j$ only depend on
$p$, $q$, $a_i$, and $a_j$, and not on
$\check x_i$ or $\check x_j$; (iii) they are always defined.
Hence, there exist unique $d_i,d_j$ such
that for any $\check \x$ such that $\check x_j = p \check x_i$, the
parametric nonlinear system \eqref{eq:qxzero}--\eqref{eq:barprod}
admits the solution $(\bar s, \bar \x) \in \mathbb R_+ \times
\Rnp$. Non-negativity of $\bar \x$, in particular of $\bar x_i$ and $\bar
x_j$, follows from $0 < \eta_i < 1$ and $\eta_j > 1$, while $\bar s\ge
0$ since $d_i<0$ and $\check x_i + \bar s d_i = \eta_i\check x_i \le
\check x_i$.
\end{proof}

Although $\bar s$ does depend on $\check \x$, the key fact is that
direction $(d_i,d_j)$ defines a
bijection from $P_{ij}$ to $Q_{ij}$ by joining pairs of points that have
same value of the function $\prod_{i\in N} x_i^{a_i}$. This suggests
that on a direction orthogonal to $(d_i,d_j)$, i.e., $(d_j, -d_i)$, we
can define a lower-bounding function that matches the values on
$P_{ij}$ and $Q_{ij}$ and that, for $n=2$, is the lower envelope of
$f(\x)$.

\begin{prop}
\label{prop:lower_betahi}
The function $f'_\ell(\x) = \lambda (d_j x_i - d_i x_j)^{a_i +
  a_j} \prod_{k\in N\setminus\{i,j\}} x_k^{a_k}$, with
$\lambda = p^{a_j} / \left(d_j - d_i p\right)^{a_i +
  a_j}$,
\begin{enumerate}[(a)]
  \item \label{prop:lowenv:match} matches the value of $f(\x)$ for
    $\x\in P_{ij}\cup Q_{ij}$;
  \item \label{prop:lowenv:minor} is a minorant of $f(\x)$ for $\x
    \in W_{ij}$;
  \item \label{prop:lowenv:enve} is, for $n=2$ and $\beta = a_1 + a_2 \ge 1$,
    the lower envelope of $f(\x)$ in $W_{ij}$.
\end{enumerate}
\end{prop}

\begin{proof}
\noindent {\bf \ref{prop:lowenv:match}}
 In order to prove $f'_\ell(\x)
= f(\x)\forall \x\in P_{ij}\cup Q_{ij}$, it suffices to prove that
\[
  \lambda (d_j x_i - d_i x_j)^{a_i + a_j} = x_i^{a_i} x_j^{a_j}
  \,\forall x_i,x_j: x_j = p x_i \vee x_j = q x_i.
\]
We prove this for $\x\in P_{ij}$ first. For $x_j = p x_i$ we have
\[
  \begin{array}{rcl}
    \lambda (d_j x_i - d_i (p x_i))^{a_i + a_j} &=&  x_i^{a_i} (p x_i)^{a_j}\\
    \lambda (d_j - p d_i)^{a_i + a_j} x_i^{a_i + a_j} &=& p^{a_j} x_i^{a_i + a_j}\\
    \lambda (d_j - p d_i)^{a_i + a_j}  &=& p^{a_j},
  \end{array}
\]
which yields $\lambda = \frac{p^{a_j}}{(d_j - p d_i)^{a_i +
    a_j}}$ in the statement. For $x_j = q x_i$ we obtain
$\lambda' = \frac{q^{a_j}}{(d_j - q d_i)^{a_i +
    a_j}}$. To prove that $\lambda = \lambda'$, observe that from
\eqref{eq:didjratio} we obtain $d_j(\eta_i - 1) = p d_i (\eta_j - 1)$,
which implies
\[
  d_j - p d_i = d_j \eta_i - p d_i \eta_j.
\]
From \eqref{eq:pqratio} we obtain $p \eta_j = q \eta_i$, which means
$d_j - p d_i = (d_j - q d_i)\eta_i$. Also,
$\eta_i^{a_i}\eta_j^{a_j} = 1$ implies $\eta_i^{a_i} =
\eta_j^{-a_j}$. Therefore,
\[
\begin{array}{llll}
  \lambda' & = & \frac{q^{a_j}}{(d_j - q d_i)^{a_i + a_j}}\\
           & = & \frac{q^{a_j}}{(d_j - p d_i)^{a_i + a_j} \eta_i^{a_i + a_j}}\\
           & = & \frac{q^{a_j}\eta_j^{a_j}}{(d_j - p d_i)^{a_i + a_j} \eta_i^{a_j}},\\
\end{array}
\]
but \eqref{eq:pqratio} implies $\frac{q^{a_j}
  \eta_j^{a_j}}{\eta_i^{a_j}} = p^{a_j}$, hence the two values of
$\lambda$ match.

\medskip

\noindent {\bf \ref{prop:lowenv:minor}}
First, $\tilde \x$ is a convex
combination $\mu \check \x + (1 - \mu) \hat \x$ of two points $\check
\x \in P_{ij}$ and $\hat \x\in Q_{ij}$ such that $f'_\ell(\check \x) =
f'_\ell(\hat \x) = f'_\ell(\tilde \x)$.
We construct $\check\x$ and $\hat\x$ as follows: $\check x_h = \hat x_h = \tilde x_h$
for $i\neq h\neq j$; then we find $\check x_i$ and $\check x_j$
by
solving the linear system $d_j \check x_i - d_i \check x_j = d_j
\tilde x_i - d_i \tilde x_j; \check x_j = p \check x_i$, and similar
for $\hat \x$.
It is easily proven that both systems admit a non-negative solution.
Hence we have to prove that $f(\tilde \x) \ge f(\check
\x) = f(\hat \x)$. Note that $\tilde x_i = \mu \check x_i +
(1-\mu)\eta_i \check x_i$ and similar for $\tilde x_j$, therefore
\[
  \begin{array}{lll}
  f(\tilde \x) &=& \prod_{k\in N\setminus \{i,j\}} \tilde x_k^{a_k} \cdot \tilde x_i^{a_i} \tilde x_j^{a_j}\\
               &=& \prod_{k\in N\setminus \{i,j\}} \check x_k^{a_k} \cdot \check x_i^{a_i} \check x_j^{a_j} \cdot (\mu + (1 - \mu)\eta_i)^{a_i} (\mu + (1 - \mu)\eta_j)^{a_j},\\
  \end{array}
\]
and we simply need to prove that $(\mu + (1 - \mu)\eta_i)^{a_i} (\mu +
(1 - \mu)\eta_j)^{a_j} \ge 1$. By concavity of the logarithm,
\[
\begin{array}{ll}
      & \log\left((\mu + (1 - \mu)\eta_i)^{a_i} (\mu + (1 - \mu)\eta_j)^{a_j}\right)\\
    = & a_i \log(\mu + (1 - \mu)\eta_i) + a_j \log(\mu + (1 - \mu)\eta_j) \\
  \ge & a_i (\mu \log(1) + (1 - \mu)\log \eta_i) +  a_j (\mu \log(1) + (1 - \mu)\log \eta_j) \\
    = & (1 - \mu) (a_i \log \eta_i + a_j \log \eta_j) \\
    = & (1 - \mu) \log(\eta_i^{a_i}\eta_j^{a_j}) = 0,
\end{array}
\]
which proves this part.
\medskip

\noindent {\bf \ref{prop:lowenv:enve}}
For $n=2$, the function
$f'_\ell(\x) = \lambda(d_j x_i - d_i x_j)^\beta$ is convex since
$\beta\ge 1$. Building on \ref{prop:lowenv:match} and
\ref{prop:lowenv:minor}, suppose another convex function $g:\mathbb
R^2\to \mathbb R$ exists such that $g(\x)\le f(\x)$ for $\x\in W_{12}$
and there exists $\tilde \x$ such that $ g(\tilde \x) > f'_\ell(\tilde
\x)$. Clearly $\tilde \x\notin (P_{ij}\cup Q_{ij})$ as otherwise
\ref{prop:lowenv:match} would imply $g(\tilde \x) > f(\tilde
\x)$. Similar to the argument in \ref{prop:lowenv:minor}, consider
then $\check \x\in P_{ij}$ and $\hat \x \in Q_{ij}$ such that $d_2
\check x_1 - d_1 \check x_2 = d_2 \tilde x_1 - d_1 \tilde x_2 = d_2
\hat x_1 - d_1 \hat x_2$, i.e., the three points are aligned on a
level curve of $f'_\ell(\x)$.  Obviously $\tilde \x$ is a convex
combination (with weight denoted by $\mu$) of $\check \x$ and $\hat
\x$ and $f(\check \x) = f(\hat \x) = f'_\ell(\check \x) = f'_\ell(\hat
\x) = f'_\ell(\tilde \x) < g(\tilde \x)$. This implies that $g(\tilde
\x) > \mu g(\check \x) + (1-\mu)g(\hat \x)$ and hence $g$ is not
convex, a contradiction.
\end{proof}

Note that case \ref{prop:lowenv:enve} of the previous proposition is
valid over the whole $W_{12}$; we are interested in the lower envelope
over $X\cap W_{12}$, which will take an extra step. We first prove a
similar result for $\beta \le 1$.

\begin{prop}
  \label{prop:lower_betalow}
  For $\beta \le 1$, the function
\[
  \textstyle f''_\ell(\x) = \zeta(d_j x_i - d_i x_j)^{\frac{a_i + a_j}{\beta}}
                       \left(\prod_{k\in N\setminus\{i,j\}}x_k^{a_k}\right)^{\frac{1}{\beta}} +
                       z_0,
\]
  with $\zeta = \lambda^{1/\beta} \frac{u-\ell}{u^{1/\beta} -
    \ell^{1/\beta}}$ and $z_0$ as defined in \eqref{eq:defzgamma},
\begin{enumerate}[(a)]
  \item \label{prop:lower_betalow:matches} matches $f(\x)$ at
    $(P_{ij}\cup Q_{ij})\cap (C_\ell\cup C_u)$;
  \item \label{prop:lower_betalow:minorant} is a minorant of $f$ in
    $X \cap W_{ij}$;
  \item \label{prop:lower_betalow:envelope} is the lower envelope of
    $f$ over $\conv((P_{ij}\cup Q_{ij})\cap (C_\ell\cup C_u))$ for
    $n=2$.
\end{enumerate}
\end{prop}

\begin{proof}
  {\bf \ref{prop:lower_betalow:matches}} From Proposition
  \ref{prop:lower_betahi} (a-b), $\lambda(d_j x_i - d_i x_j)^{a_i +
    a_j}\prod_{k\in N\setminus\{i,j\}} x_k^{a_k}$ is a minorant of $f$
  within the same domain and matches $f$ in $P_{ij}\cup
  Q_{ij}$. Hence, for $\x\in P_{ij}\cup Q_{ij}$ such that
  $f(\x)=\ell$ we have $\lambda(d_j x_i - d_i x_j)^{a_i +
    a_j}\prod_{k\in N\setminus\{i,j\}} x_k^{a_k}= \ell$, i.e., $(d_j x_i - d_i x_j)^{\frac{a_i +
      a_j}{\beta}}\left(\prod_{k\in N\setminus\{i,j\}} x_k^{a_k}
  \right)^{\frac{1}{\beta}}= \left( \frac{\ell}{\lambda}
  \right)^{\frac{1}{\beta}}$, and similar for $u$. Therefore
\[
  \begin{array}{ll}
    \zeta\left(\frac{\ell}{\lambda}\right)^{\frac{1}{\beta}} + \zeta_0 = \ell\\
    \zeta\left(\frac{u}{\lambda}\right)^{\frac{1}{\beta}} + \zeta_0 = u,\\
  \end{array}
\]
which is a linear system in $\zeta, \zeta_0$ with solutions
\[
  \begin{array}{l}
  \zeta = \frac{u-\ell}{\left(\frac{u}{\lambda}\right)^{\frac{1}{\beta}} - \left(\frac{\ell}{\lambda}\right)^{\frac{1}{\beta}}}
        = \lambda^{\frac{1}{\beta}} \frac{u-\ell}{u^{\frac{1}{\beta}} - \ell^{\frac{1}{\beta}}} = (\gamma \lambda)^{\frac{1}{\beta}}\\
  \zeta_0 = \ell - \zeta\left(\frac{\ell}{\lambda}\right)^{1/\beta}
  = \ell - \frac{ (u-\ell)\ell^{1/\beta}}{u^{1/\beta} -
    \ell^{1/\beta}}=\frac{u^{1/\beta}\ell -\ell^{1/\beta}u}{u^{1/\beta} -
    \ell^{1/\beta}} = z_0.
  \end{array}
\]
Therefore, $f''_\ell(\x) = f'_\ell(\x) = f(\x)$ for $\x$ such that
$x_j/x_i \in \{p, q\}, f(\x) \in \{\ell,u\}$.
\medskip

\noindent {\bf \ref{prop:lower_betalow:minorant}}
By Lemma \ref{lemma:containment}, if $\beta \le 1$ then $f_u(\x) = z_0 +
\left(\gamma \prod_{k\in N}x_k^{a_k}\right)^{\frac{1}{\beta}}$ is a
minorant of $f(\x)$ for $f(\x)\in [\ell,u]$. Therefore it suffices to
prove that $f''_\ell(\x) \le f_u(\x)$ in $X\cap W_{ij}$.
\[
  \begin{array}{lrlr}
                    & z_0 + \zeta(d_j x_i - d_i x_j)^{\frac{a_i + a_j}{\beta}}\left(\prod_{k\in N\setminus\{i,j\}}x_k^{a_k}\right)^{\frac{1}{\beta}} &\le& z_0 + \left(\gamma \prod_{k\in N}x_k^{a_k}\right)^{\frac{1}{\beta}}\\
    \Leftrightarrow & (\gamma\lambda)^{\frac{1}{\beta}}(d_j x_i - d_i x_j)^{\frac{a_i + a_j}{\beta}}\left(\prod_{k\in N\setminus\{i,j\}}x_k^{a_k}\right)^{\frac{1}{\beta}}          &\le& \left(\gamma \prod_{k\in N}x_k^{a_k}\right)^{\frac{1}{\beta}}\\
    \Leftrightarrow & \lambda^{\frac{1}{\beta}}(d_j x_i - d_i x_j)^{\frac{a_i + a_j}{\beta}}\phantom{\left(\prod_{k\in N\setminus\{i,j\}}x_k^{a_k}\right)^{\frac{1}{\beta}}}          &\le& \left(x_i^{a_i}x_j^{a_j}\right)^{\frac{1}{\beta}}.
  \end{array}
\]
The last inequality is equivalent to $\lambda(d_j x_i - d_i x_j)^{a_i
  + a_j} \le x_i^{a_i}x_j^{a_j}$, which holds in $X\cap W_{ij}$ for
Proposition \ref{prop:lower_betahi}. Although it is applied here to
the two-variable monomial $x_i^{a_i} x_j^{a_j}$ rather than $f(\x)$,
the relationship still holds because $d_i, d_j$ only depend on
$p,q,a_i,a_j$, and $\lambda$ depends on $d_i,d_j,p,q,a_i,a_j$.
\medskip

\noindent {\bf \ref{prop:lower_betalow:envelope}}
For $n=2$, $f''_{\ell}(\x)$ reduces to the linear function $\zeta(d_2
x_1 - d_1 x_2) + z_0$.
By \ref{prop:lower_betalow:matches} and
\ref{prop:lower_betalow:minorant}, $f''_\ell$ matches the concave
function $f(\x)$ at all four points of the set $T = (P_{ij}\cup
Q_{ij}) \cap (C_\ell \cup C_u)$. For linearity of $f''_\ell$, any
function $g$ defined on $\conv(T)$ such that $g(\tilde \x) >
f''_\ell(\tilde \x)$ for some $\tilde \x\in \conv(T)$ is concave
w.r.t.\mbox{} two or more elements of $T$, and hence cannot be the
(convex) lower envelope of $f$ over $\conv(T)$.
%
%
\end{proof}

\begin{lemma}
  \label{lemma:conv_hull_xwij}
  For $n=2$, the convex hull of $X \cap W_{12}$ is
  \[
    \begin{array}{ll}
      Y =\!\!& \{\x\in \mathbb R^2_+:\\
         & \phantom{\{} p x_1 \le x_2 \le q x_1,\\
         & \phantom{\{} x_1^{a_1} x_2^{a_2} \ge \ell,\\
         & \phantom{\{} d_2 x_1 - d_1 x_2 \le (u/\lambda)^{\frac{1}{\beta}}\}.
    \end{array}
  \]
\end{lemma}

\begin{proof}
  $\lambda(d_2 x_1 - d_1 x_2)^\beta \le f(\x)$ and $f(\x)\le u$
  imply $d_2 x_1 - d_1 x_2 \le (u/\lambda)^{\frac{1}{\beta}}$ and
  consequently $X \cap W_{12} \subseteq Y$. By convexity of $Y$ we
  also have $\conv(X\cap W_{12}) \subseteq Y$. In order to prove
  $\conv(X\cap W_{12}) \supseteq Y$, consider $\tilde \x\in Y$. If
  $\tilde \x\in X\cap W_{12}$, the result holds. Otherwise $f(\tilde
  \x) > u \ge \lambda(d_2 \tilde x_1 - d_1 \tilde
  x_2)^\beta$. Consider two vectors $\check \x\in P_{12}$ and $\hat \x
  \in Q_{12}$ such that $d_2 \check x_1 - d_1 \check x_2 = d_2 \hat
  x_1 - d_1 \hat x_2 = d_2 \tilde x_1 - d_1 \tilde x_2$. Such points
  satisfy the two systems
  \[
  \left\{
    \begin{array}{l}
      d_2 \check x_1 - d_1 \check x_2 = d_2 \tilde x_1 - d_1 \tilde x_2\\
      p \check x_1 - \check x_2 = 0
    \end{array}
  \right.
  \qquad
  \left\{
    \begin{array}{l}
      d_2 \hat   x_1 - d_1 \hat   x_2 = d_2 \tilde x_1 - d_1 \tilde x_2\\
      q \hat x_1 - \hat x_2 = 0,
    \end{array}
  \right.
  \]
both of which always yield a solution as $d_2,p,q > 0$ and $d_1 <
0$. Because $\check \x \in P_{12}$, $\hat \x\in Q_{12}$, and $f'_\ell(\x) =
\lambda(d_2 x_1 - d_1 x_2)^\beta = \lambda(d_2 \tilde x_1 - d_1 \tilde
x_2)^\beta \le u$ for both $\check \x$ and $\hat \x$, these belong to
$X\cap W_{12}$ and $\tilde \x$ is their convex combination since
$\frac{\check x_2}{\check x_1} = p \le \frac{\tilde x_2}{\tilde x_1}
\le q = \frac{\hat x_2}{\hat x_1}$, proving that $\tilde \x \in
\conv(X\cap W_{12})$.
\end{proof}

The last condition defining $Y$, i.e., $d_2 x_1 - d_1 x_2 \le
(u/\lambda)^{\frac{1}{\beta}}$, is equivalent to setting an upper
bound of $u$ on the lower-bounding function $f'_\ell(\x)$ for
$\beta\ge 1$, as it is equivalent to $\lambda (d_2 x_1 - d_1
x_2)^\beta \le u$. It is also equivalent to setting an upper bound of
$u$ on the lower-bounding function $f''_\ell(\x)$ for the case $\beta
\le 1$: $f''_\ell(\x) = \zeta(d_2 x_1 - d_1 x_2) + z_0 \le u$ if
$d_2 x_1 - d_1 x_2 \le \frac{u-z_0}{\zeta}$, but
\[
  \frac{u-z_0}{\zeta} =
  \frac{u-z_0}{\lambda^{\frac{1}{\beta}}\frac{u-\ell}{u^{\frac{1}{\beta}} - \ell^{\frac{1}{\beta}}}} =
  \frac{u^{\frac{1}{\beta}} - \ell^{\frac{1}{\beta}}}{\lambda^{\frac{1}{\beta}}(u-\ell)}
  \left(u - \frac{u^{\frac{1}{\beta}}\ell - \ell^{\frac{1}{\beta}}u}{u^{\frac{1}{\beta}} - \ell^{\frac{1}{\beta}}}\right) =
  \frac{u^{\frac{1}{\beta}}}{\lambda^{\frac{1}{\beta}}}.
\]

Because $f$ is defined over  $\Rnp$, Lemma
\ref{lemma:conv_hull_xwij} defines the projection onto the $\x$-space
of the lower and upper envelopes of $f$ over $X\cap W_{12}$ and
consequently the projection of $\textrm{conv}(F(W_{12}))$. The
argument used in the above proof is worth underlining and will be used
below: a vector $\tilde \x \in Y$ such that $f(\tilde \x) > u \ge
f'_\ell(\tilde \x)$ is a member of $\conv(X\cap W_{12})$.

\def\expAI{1.7}
\def\expAJ{1.5}

\def\pa{0.35} 
\def\qa{3}

\def\Zlb{0.4}
\def\Zub{10}


\begin{figure}[t]
  \centering
  \subfigure[The upper envelope $f'_u(x_1,x_2) = z_0
    + (\gamma x_1^{a_1}x_2^{a_2})^{\frac{1}{\beta}}$.]{
    \includegraphics[width=.48\textwidth]{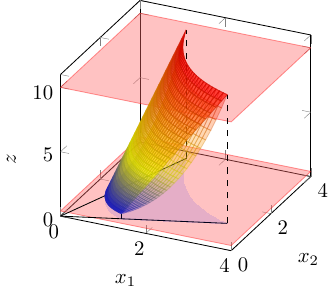}
    \includegraphics[width=.48\textwidth]{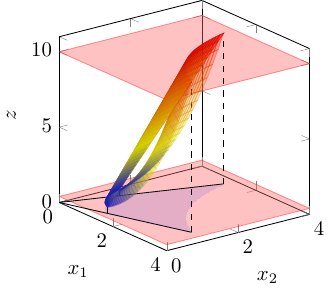}
  }
  \subfigure[The lower envelope $f'_\ell(x_1,x_2) = \lambda (d_2 x_1 - d_1 x_2)^{\beta}$.]{
    \includegraphics[width=.48\textwidth]{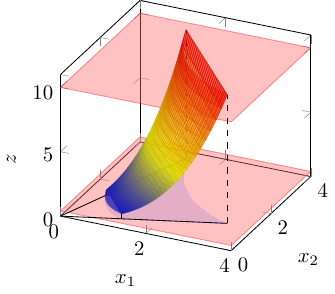}
    \includegraphics[width=.48\textwidth]{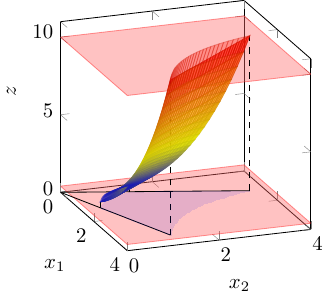}
  }
  \caption{Upper envelope $f'_u(\x)$ (top) and lower envelope
    $f'_\ell(\x)$ of function $f(\x) = x_1^{a_1}x_2^{a_2}$ for $\beta \ge 1$,
    both shown under different angles. The domain $\{(x_1,x_2)\in
    \mathbb R_+^2: p x_1 \le x_2 \le q x_1, f(\x)\in[\ell,u]\}$ is the
    shaded area on the $x_1,x_2$ plane. The parameters are as follows:
    $(p,q) = (\pa,\qa)$, $(\ell, u) = (\Zlb, \Zub)$, $(a_1,
    a_2) = (\expAI, \expAJ)$.}
  \label{fig:3denvelopes}
\end{figure}

We are now ready to provide the {\em upper} envelope of $f$ over $X \cap
W_{ij}$ for $n=2$, which was left uncovered from the previous
section. Figures \ref{fig:3denvelopes} and  \ref{fig:3denvelopesLowBeta} show upper and lower
envelopes for sample values of the parameters. The proof is,
up to Case 1 included, similar to the proof of Proposition
\ref{prop:upper_env} for $n>2$.

\begin{thm}
  \label{thm:upper_env_n2_bh}
  For $n=2$, the upper envelope $E_U(f, X \cap W_{12})$ is
  \[
   \begin{array}{rcl}
     &H =\!\! & \{(\x,z)\in Y\times \mathbb R: \\
     && z \le f_u(\x)\\
     && z \le u\},
   \end{array}
   \]
  where $f_u(\x) = f'_u(\x) = z_0 + \left(\gamma x_1^{a_1}
  x_2^{a_2}\right)^{\frac{1}{\beta}}$ if $\beta \ge 1$ and $f_u(\x) =
  f(\x)$ otherwise.
\end{thm}

\begin{proof}
  Any $(\x,z)\in \hypg(f, X \cap W_{ij})$ satisfies $\x\in Y$ for Lemma
  \ref{lemma:conv_hull_xwij} and $z\le u$ as $\x\in X$. Also, $(\x,z)$ satisfies
  $z \le f_u(\x)$: by construction for $\beta \le 1$, and by
  Lemma \ref{lemma:containment} for $\beta \ge 1$. Hence $\hypg(f,X
  \cap W_{ij}) \subseteq H$.
  Because $H$ is convex, $E_U(f,X \cap W_{ij}) \subseteq H$. We now
  prove that $H \subseteq E_U(f,X \cap W_{ij})$.  Consider $(\tilde
  \x, \tilde z)\in H$. Two cases arise: either $\tilde x_1^{a_1}\tilde
  x_2^{a_2} \le u$ or $\tilde x_1^{a_1}\tilde x_2^{a_2} > u$.

  \subparagraph*{Case 1: $\tilde x_1^{a_1} \tilde x_2^{a_2} \le u$.}
  If $\beta \le 1$ or $\tilde z \le \tilde x_1^{a_1} \tilde
  x_2^{a_2}$, then obviously $(\tilde \x, \tilde z)\in \hypg(f,X\cap
  W_{12})$ and the result holds. Otherwise, see Case 1, proof of
  Proposition \ref{prop:upper_env}.

  \subparagraph*{Case 2: $\tilde x_1^{a_1} \tilde x_2^{a_2} > u$.}
  There exist $\check \x = s'\tilde\x$ such that $f(\check \x) =
  f_u(\check \x) = u$ (again, by construction for $\beta \le 1$ and by
  Lemma \ref{lemma:containment} for $\beta \ge 1$) and $\hat \x =
  s''\tilde \x$ such that $\lambda (d_2 \hat x_1 - d_1 \hat x_2)^\beta
  = f'_\ell (\hat \x) = u \le f(\hat \x)$.
  Because $\tilde \x$ satisfies $\lambda(d_2 \tilde x_1 - d_1 \tilde
  x_2)^\beta \le u$ and $f'_\ell$ is monotone increasing along the
  direction $\{s\tilde \x: s \ge 0\}$, $s''$ is unique and $s''\ge
  1$. Similarly, $s'$ is unique and $s' < 1$. Then
  $\tilde \x$ is a convex combination of $\check \x$ and $\hat \x$.

  Whereas $(\check \x,z) \in \hypg(f,X \cap W_{ij})$ for $z\le u$,
  $\hat \x$ is in turn a convex combination of two points in $X \cap
  W_{ij}$: $\check{\x}$ and $\hat{\x}$ such that $f(\check{\x})=f(\hat{\x})=u$, $x^p_2 = p
  x^p_1$, and $x^q_2 = q x^q_1$. Properties of $f'_\ell$ show that
  $f'_\ell(\check{\x}) = f'_\ell(\hat{\x}) = f'_\ell(\hat \x) = u$, which
  implies $\check{\x}, \hat{\x}, \hat \x$ lie on the same level line of
  $f'_\ell$ and $d_2 x^p_1 - d_1 x^p_2 = d_2 x^q_1 - d_1 x^q_2 = d_2
  \check x_1 - d_1 \check x_2 =
  \left(\frac{u}{\lambda}\right)^{\frac{1}{\beta}}$. Taking into
  account $z$, the above implies implies $(\tilde \x, u)$ is a convex
  combination of $(\check \x, u)$, $(\check{\x}, u)$, and $(\hat{\x}, u)$, all
  elements of $X \cap W_{12}$. But then $(\check \x, \tilde z)$,
  $(\check{\x}, \tilde z)$, and $(\hat{\x}, \tilde z)$ all belong to $\hypg(f,X
  \cap W_{ij})$ by definition of hypograph, and because $\tilde z \le
  u$, $(\tilde \x,\tilde z)$ is their convex combination hence
  $(\tilde \x, \tilde z)\in E_U(f,X \cap W_{ij})$.
\end{proof}

\begin{thm}
  \label{thm:lower_env_n2}
  For $n=2$, the lower envelope $E_L(f, X \cap W_{12})$ is
  \[
    \begin{array}{rcl}
      &H =\!\! & \{(\x, z)\in Y\times \mathbb R:\\
      && z \ge f_\ell(\x),\\
      && z \ge \ell\},
    \end{array}
  \]
where $f_\ell(\x) = f'_\ell(\x)$ if $\beta \ge 1$ and $f_\ell(\x) =
f''_\ell(\x)$ if $\beta \le 1$.
\end{thm}

\begin{proof}
Any
$(\tilde \x, \tilde z) \in \textrm{epi}(f,X\cap W_{12})$ satisfies
$\x\in Y$ and $z \ge \ell$ by definition. It also satisfies
$z \ge f'_\ell(\x)$ by Proposition \ref{prop:lower_betahi} and $z \ge f''_\ell(\x)$
by Proposition \ref{prop:lower_betalow}. Hence
$\textrm{epi}(f,X\cap W_{12})\subseteq H$. By convexity of
$H$ we conclude $E_L(f,X\cap W_{12})\subseteq H$.
We prove now $H\subseteq E_L(f, X\cap W_{12})$.

Consider $(\tilde \x, \tilde z)\in H$. If $\tilde \x \in
\textrm{conv}((C_\ell \cup C_u) \cap (P_{12} \cup Q_{12}))$, then for
$\beta \le 1$ the vector $(\tilde \x, f''_\ell(\x))$ is a convex
combination of the four vectors $(\dot \x, f''_\ell(\dot \x)) = (\dot
\x, f(\dot \x))$ for $\dot \x \in (C_\ell \cup C_u) \cap (P_{12} \cup
Q_{12})$, given that $f''_\ell$ is linear (see Proposition
\ref{prop:lower_betalow}). For $\beta \ge 1$, $f'_\ell$ matches $f$
within $P_{12}\cup Q_{12}$ (see Proposition \ref{prop:lower_betahi}),
hence there exist $\hat \x, \check \x$ such that $(\tilde \x,
f'_\ell(\x))$ is a convex combination of $(\hat \x, f'_\ell(\hat \x))$
and $(\check \x, f'_\ell(\check \x))$, both in $\epig(f, X \cap
W_{12})$. In both cases, $z \ge f_\ell(\x)$ and therefore $(\tilde \x,
\tilde z)\in E_L(f, X \cap W_{12})$.

If $\tilde \x \in Y \setminus \textrm{conv}((C_\ell \cup C_u) \cap
(P_{12} \cup Q_{12}))$, then $f(\x) \ge \ell \ge f_\ell(\x)$. A
similar argument to the above leads to three points $\check \x \in
C_\ell \cap P_{12}$, $\hat \x \in C_\ell \cap Q_{12}$ (these two sets
contain a single vector each), and $\dot \x$, such that $f(\dot \x) =
f(\check \x) = f(\hat \x) = \ell$ and that form $(\tilde \x, \ell)$ as
their convex combination, where $\tilde z \ge \ell$, yielding $(\tilde
\x, \tilde z)\in E_L(f, X \cap W_{12})$. This concludes the proof;
note that $Y \setminus \textrm{int}(\textrm{conv}((C_\ell \cup C_u) \cap (P_{12}
\cup Q_{12})))$ is the convex hull of $C_\ell \cap W_{12}$.
\end{proof}





\def\expEI{0.1}
\def\expEJ{0.2}
\def\pe{0.4}
\def\qe{3.3}
\def\ZlbE{0.65}
\def\ZubE{1.21}

\begin{figure}[t]
  \centering
  \includegraphics[width=.48\textwidth]{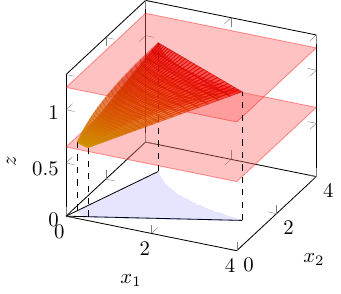}
  \includegraphics[width=.48\textwidth]{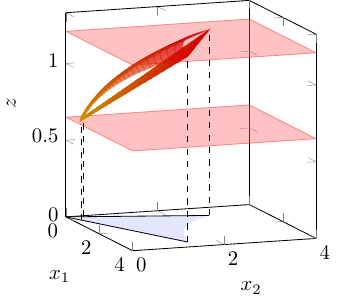}
  \caption{Lower envelope $f''_\ell(\x)$ of function $f(\x) =
    x_1^{a_1}x_2^{a_2}$ shown under different angles. The domain
    $\{(x_1,x_2)\in \mathbb R_+^2: p x_1 \le x_2 \le q x_1,
    f(\x)\in[\ell,u]\}$ is the shaded area on the $x,y$ plane. The
    parameters are as follows: $(p,q) = (\pe,\qe)$, $(\ell, u) =
    (\ZlbE, \ZubE)$, $(a_1, a_2) = (\expEI, \expEJ)$.}
  \label{fig:3denvelopesLowBeta}
\end{figure}

Given a function
$f$ and its lower and upper envelopes $f''_\ell$ and $f''_u$ over a set
$\domn$, the convex hull of $\{(\x,z)\in \mathbb R^{n+1}: z = f(\x)\}$
is given by the intersection of lower and upper envelope
\citep{tawarmalani2013decomposition}. The convex hull of $X \cap
W_{ij}$ for $n=2$ is hence the intersection of the upper and lower
envelope described above.

Unlike the convex hull for $n=2$, the one for $n>2$ has one key
feature: its projection onto the $\x$ variables is unbounded. For this reason the
extra linear inequality $ d_2 x_1 - d_1 x_2 \le
\left(\frac{u}{\lambda}\right)^{\frac{1}{\beta}}$ is only valid for
$n=2$.

\begin{corollary}
\label{cor:convex_hull}
The convex hull of $F(W_{12})$ for $n=2$ and $\beta \ge 1$ is
\[
\begin{array}{rl}
  \conv(F(W_{12})) =\!& \{(\x,z)\in Y \times \mathbb R:\\
                         & \max\{\ell, \lambda (d_2 x_1 - d_1 x_2)^\beta\}
                           \le z \le
                           \min \{u, z_0 + (\gamma x_1^{a_1} x_2^{a_2})^{\frac{1}{\beta}}\}\}.
  \end{array}
\]
The convex hull of $F(W_{12})$ for $n=2$ and $\beta \le 1$ is
\[
\begin{array}{rl}
  \conv(F(W_{12})) =\!& \{(\x,z)\in Y \times \mathbb R:\\
                         & \max\{\ell, \zeta (d_2 x_1 - d_1 x_2) + z_0\}
                           \le z \le
                           \min \{u,  x_1^{a_1} x_2^{a_2}\}\}.
  \end{array}
\]
\end{corollary}

For $\beta \ge
1$, the upper envelope $f'_u$ matches $f$ only on the level sets of
$f$, i.e., $C_\ell \cup C_u = \{\x\in \mathbb R_+^2: f(x)\in
\{\ell,u\}\}$, and the lower envelope $f'_\ell$ matches $f$ on the set
$P \cup Q = \{\x\in \mathbb R_+^2:x_2=p x_1 \vee x_2 = q x_1\}$.
For $\beta \le 1$, the upper envelope $f''_u$ matches $f$ in $(P \cup Q)
\cup (C_\ell \cup C_u)$, while the lower envelope $f''_\ell$ matches $f$ in
four points: those of the set $(P_{12} \cup Q_{12}) \cap (C_\ell \cup C_u)$.

Table \ref{tab:summary} summarizes the results from this and the
previous two sections. The lower envelope for $W_{ij}$ is an
open problem for $n>2$.

\begin{table}[t]
  \centering
  \begin{tabular}{l|c|c}
    \hline
    & $\beta \ge 1$ & $\beta \le 1$ \\
    \hline
    & \multicolumn{2}{|c}{$n\ge 2$}\\
    \hline
    $\conv(F(\Rnp))$
    & $
      \begin{array}{c}
        \x\in\Rnp\\
        (z-z_0)^\beta \le \gamma \prod_{k\in N} x_k^{a_k}\\
        z\in[\ell,u]
      \end{array}
      $
    & $
      \begin{array}{c}\x\in\Rnp\\
        z \le \prod_{k\in N} x_k^{a_k}\\
        z\in[\ell,u]
      \end{array}
      $\\
    \hline
    $E_U(f, X\cap W_{ij})$
    & $
      \begin{array}{c}
        \x\in\Rnp\\
        (z-z_0)^\beta \le \gamma \prod_{k\in N} x_k^{a_k}\\
        z \le u\\
        p x_i \le x_j \le q x_i\\
        \prod_{k\in N} x_k^{a_k} \ge \ell
      \end{array}
      $
    & $
      \begin{array}{c}
        \x\in\Rnp\\
        z \le \prod_{k\in N} x_k^{a_k}\\
        z \le u\\
        p x_i \le x_j \le q x_i\\
        \prod_{k\in N} x_k^{a_k} \ge \ell
      \end{array}
      $\\
    \hline
    & \multicolumn{2}{|c}{$n=2$}\\
    \hline
    $E_L(f, X\cap W_{ij})$
    & $
      \begin{array}{c}
        p x_1 \le x_2 \le q x_1\\
        x_1^{a_1} x_2^{a_2} \ge \ell\\
        d_2 x_1 - d_1 x_2 \le (u/\lambda)^{\frac{1}{\beta}}\\
        z \ge \lambda(d_2 x_1 - d_1 x_2)^\beta\\
        z \ge \ell\\
      \end{array}
      $
    &
      $
      \begin{array}{c}
        p x_1 \le x_2 \le q x_1\\
        x_1^{a_1} x_2^{a_2} \ge \ell\\
        d_2 x_1 - d_1 x_2 \le (u/\lambda)^{\frac{1}{\beta}}\\
        z \ge \zeta(d_2 x_1 - d_1 x_2) + z_0\\
        z \ge \ell\\
      \end{array}
      $
      \\
    \hline
    $E_U(f, X\cap W_{ij})$
    & $
      \begin{array}{c}
        p x_1 \le x_2 \le q x_1\\
        x_1^{a_1} x_2^{a_2} \ge \ell\\
        d_2 x_1 - d_1 x_2 \le (u/\lambda)^{\frac{1}{\beta}}\\
        (z-z_0)^\beta \le \gamma x_1^{a_1} x_2^{a_2}\\
        z \le u\\
      \end{array}
      $
    & $
      \begin{array}{c}
        p x_1 \le x_2 \le q x_1\\
        x_1^{a_1} x_2^{a_2} \ge \ell\\
        d_2 x_1 - d_1 x_2 \le (u/\lambda)^{\frac{1}{\beta}}\\
        z \le x_1^{a_1} x_2^{a_2}\\
        z \le u\\
      \end{array}
      $\\
    \hline
  \end{tabular}
  \caption{Summary of results on lower and upper envelopes
    of $f$ over $X$ and over $X\cap W_{ij}$.}
  \label{tab:summary}
\end{table}

\section{Volume of the convex hull}
\label{sec:volume}

Computing the volume of the convex hull of
$f(\x)$ within $X \cap W_{12}$ finds a practical application in choosing
branching rules in a MINLO solver.
Branching decisions affect the efficiency of any BB algorithm in a
strong and often unpredictable way. Even for the most standard
branching decision $(x_k\le b')\vee (x_k\ge b'')$, there is no
supporting evidence of any single best branching policy for choosing the
branching variable $x_k$ \citep{achkm2005, benichou.et.al:71} or
the branching point $b',b''$ \citep{belottillmw09}.

The volume of the convex hull of nonlinear operators has been considered for bi- and
tri-linear terms $x_i x_j$ and $x_i x_j x_k$
\citep{lee2018algorithmic,speakman2017quantifying,speakman2017experimental,lee2019gaining},
as a measure for quality of relaxations and for evaluating
branching rules \citep{speakman2018branching}--see also
\citet{anstreicher2021convex}.

Mixed
Integer Linear Optimization (MILO) algorithms only concern
themselves with choosing a branching variable at each node, but this
is not the case for MINLO problems. Once a branching variable $x_k$ is
chosen, {\em balanced branching} attempts to create balanced BB
subtrees by selecting a branching point such that the convex hulls of
the feasible set of
each of the two resulting subproblems have equal volume, in the hope of
creating two balanced BB subtrees  \citep{sah1996,belottillmw09}. An alternative branching criterion
is that the resulting total volume of the convex hulls of the two
subproblems is minimum. Rather than computing the volume of said
convex hulls of the new feasible regions, it is more practical to compute it
for the operators containing the branched-on variable.

We assume $n=2$ from now on.
The feasible region for a two-variable monomial term $z=x_1^{a_1} x_2^{a_2}$
considered here is defined by bounds $z\in [\ell,u]$ on the monomial
itself and by the $p,q$ parameters defining the homogeneous
inequalities $p x_1 \le x_2 \le q x_1$ delimiting $x_1$ and
$x_2$. Branching on $z$ and on $x_1/x_2$ allows for maintaining tight relaxations
by using the convex envelopes described in this article. Therefore one
can consider two branching rules: given $r\in (p,q)$, either branch using $(p
x_1 \le x_2 \le r x_1) \vee (r x_1 \le x_2 \le q x_1)$ or using $(z \le \nu) \vee (z
\ge \nu)$.

Two questions arise: once the monomial term has
been chosen for branching, do we branch on $z$ or on
$\frac{x_2}{x_1}$?  Also, what is the branching point $\nu$ or $r$
that minimizes the total volume in the two subproblems or,
alternatively, finds the most balanced pair of subproblems?


\begin{figure}[t]

  \def\levelEx{4.4}
  \def\levelExE{1.05}

  \centering

  \subfigure[$(p,q)=(\pa,\qa)$, $(a_1,a_2)=(\expAI,\expAJ)$,
    $(\ell,u)=(\Zlb,\Zub)$, $z = \levelEx$]{
    \includegraphics[width=.46\textwidth]{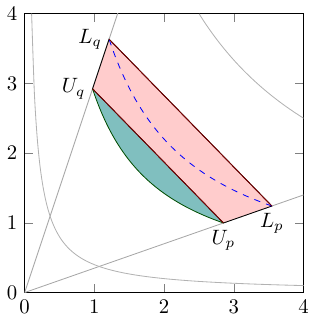}
  }
\quad
  \subfigure[$(p,q)=(\pe,\qe)$, $(a_1,a_2)=(\expEI,\expEJ)$,
    $(\ell,u)=(\ZlbE,\ZubE)$, $z = \levelExE$]{
    \includegraphics[width=.46\textwidth]{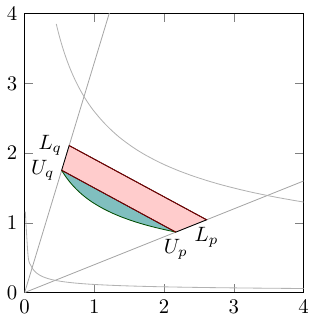}
  }
  \caption{Cross-section (two shades) of the convex envelope of
    $f(\x)=x_1^{a_1}x_2^{a_2}$ for $z\in [\ell,u]$ in the region
    $W_{12}=\{(x_1,x_2)\in \mathbb R_+^2: p x_1 \le x_2 \le q x_1\}$,
    for $\beta \ge 1$ (left) and $\beta \le 1$.
    For $\beta \ge 1$, the level curve of $f(\x)$ is the dashed arc
    between $L_p$ and $L_q$. The arc between $U_p$ and $U_q$ is the level
    curve of the upper envelope $\hat f(\x) = z_0 + \left(\gamma
    x_1^{a_1}x_2^{a_2}\right)^{\frac{1}{\beta}}$ (left) and the level
    set for $f(\x)$ (right). Finally, the segment $L_pL_q$
    is the level curve of the lower envelope $f'_\ell(\x) = \lambda(d_2
    x_1 - d_1 x_2)^{\beta} = z$ (left) and $f''_\ell(\x) = \zeta (d_2
    x_1 - d_1 x_2) + z_0 = z$ (right), for the given value
    of $z$.}
  \label{fig:xsection}
\end{figure}

Given an initial set defined by $p,q,\ell,u,a_1,a_2$, the volume of the
convex envelope denoted as $V(p,q,\ell,u)$, a branching rule on
$x_j/x_i$ using $r$ as a ratio results in a total volume of $V_r =
V(p,r,\ell,u) + V(r,q,\ell,u)$, while a branching rule on $z$ with
branching value $\nu$ results in a total volume of $V_b(\nu) =
V(p,q,\ell,\nu)+ V(p,q,\nu,u)$. If the minimum total volume is sought,
one must select between two minimizers: $r \in
\textrm{argmin}\{V_r(r):r\in [p+\epsilon,q-\epsilon]\}$ and $\nu \in
\textrm{argmin} \{V_b(\nu): \nu \in [\ell+\epsilon, u-
  \epsilon]\}$. Here $\epsilon$ is used to obtain tighter bound
interval on $\nu$ and $r$, which ensures that the feasible regions of both new subproblems are
strict subsets of the subproblem being branched on. Balanced branching requires instead $V(p,r,\ell,u) =
V(r,q,\ell,u)$ and $V(p,q,\ell,\nu) = V(p,q,\nu,u)$
respectively. Therefore, an analytical form of the volume of the
convex hull described in Corollary \ref{cor:convex_hull} would be useful.

Both envelopes of $f$ over $X\cap W_{12}$, i.e., $\max\{\ell,
f_\ell(\x)\}$ (lower) and $\min\{u, f_u(\x)\}$ (upper), are non-smooth
as $f_\ell$ and $f_u$ have non-null gradient in all of $W_{12}$
(except for $f'_\ell$ at $\x=0$). The domain $\domn$ makes it
additionally awkward to consider each envelope separately.

Consider any  $z$ in the range
$[\ell,u]$ of $z$. The volume differential $dV$ at $z$ is the
product of $dz$ and the area of the cross-section of the convex
envelope at $z$. Integrating the volume differential yields
the volume of the convex hull. The main advantage of this method is
that the structure of the cross-section only depends on $z$ and
the original parameters of the problem.

Figure \ref{fig:xsection} illustrates the cross-section as a shaded
region, which we subdivide into two  areas (indicated by different shades) for ease
of computation. The shape of the cross-section is roughly the same for
$\beta \ge 1$ and $\beta \le 1$. The segment $L_pL_q$ is the level
curve of the lower envelope $f'_\ell(\x)$ (for $\beta \ge 1$) and
$f''_\ell(\x)$ (for $\beta \le 1$), while the arc between $U_p$ and
$U_q$ is the level curve of the upper envelope $f_u(\x)$. Together
with the lines $x_2 = p x_1$ and $x_2 = q x_1$ defining $P_{12}$ and $Q_{12}$,
these delimit the cross-section. The dashed arc between $L_p$ and
$L_q$ is the level curve of $f(\x)$ at $z$ for $\beta \ge 1$ (see
Figure \ref{fig:xsection}a) and  is only shown
for completeness; it will not be used in our derivation.

As proven in Section \ref{sec:lower_env}, the extremes $U_p, U_q$ of
the arc also form a segment that is parallel to $L_p L_q$. The area of
the cross-section is hence the sum of two areas: that of the
lightly-shaded trapezoid with vertices $L_p, L_q, U_q, U_p$ and the
dark-shaded area between the arc and the segment both delimited by
$U_p,U_q$.

\subparagraph*{Trapezoid.}
The area of the trapezoid $L_pL_qU_qU_p$ is
$\frac{1}{2}\left(\Delta_{\textsc{u}} + \Delta_{\textsc{l}}\right)\delta$, where $\Delta_{\textsc{u}}$
and $\Delta_{\textsc{l}}$ are the lengths of $U_pU_q$ and $L_pL_q$, respectively,
and $\delta$ is the distance between these two parallel segments. We
need $x_1,x_2$ coordinates of all four points: $L_p=(x_1^{\textsc{l} p},
x_2^{\textsc{l} p})$, $L_q=(x_1^{\textsc{l} q}, x_2^{\textsc{l} q})$, $U_p=(x_1^{\textsc{u} p}, x_2^{\textsc{u} p})$,
$U_q=(x_1^{\textsc{u} q}, x_2^{\textsc{u} q})$.
Then $\Delta_{\textsc{l}}$ and $\Delta_{\textsc{u}}$ are $\|\x^{\textsc{l} p} -
\x^{\textsc{l} q}\|_2$ and $\|\x^{\textsc{u} p} - \x^{\textsc{u} q}\|_2$ respectively.
\[
\begin{array}{rl}
  \Delta_{\textsc{l}} &= \left((x_1^{\textsc{l} p} - x_1^{\textsc{l} q})^2 + (px_1^{\textsc{l} p} - qx_1^{\textsc{l} q})^2)\right)^{\frac{1}{2}}\\
           &= \left((1 + p^2)(x_1^{\textsc{l} p})^2 + (1 + q^2)(x_1^{\textsc{l} q})^2 -2(1 + pq) x_1^{\textsc{l} p} x_1^{\textsc{l} q}\right)^{\frac{1}{2}}\\
           &= x_1^{\textsc{l} p} \left(1 + p^2 + \eta_1^2(1 + q^2) -2 \eta_1(1+pq)\right)^{\frac{1}{2}}\\
           &= x_1^{\textsc{l} p} \left(1 + p^2 + \eta_1^2 + \eta_1^2 q^2 -2 \eta_1 -2 \eta_1 pq\right)^{\frac{1}{2}}\\
           &= x_1^{\textsc{l} p} \left((1 - \eta_1)^2 + p^2  + \eta_2^2 p^2 -2 \eta_2 p^2\right)^{\frac{1}{2}}\\
           &= x_1^{\textsc{l} p} \left((1 - \eta_1)^2 + p^2(1 - \eta_2)^2\right)^{\frac{1}{2}}.
\end{array}
\]
Define $\tau = \sqrt{(1 - \eta_1)^2 + p^2(1 -
  \eta_2)^2}$. Then $\Delta_{\textsc{l}} = \tau x_1^{\textsc{l} p}$ and $\Delta_{\textsc{u}} =
  \tau x_1^{\textsc{u} p}$.

In order to compute $\delta$, consider the equations of the line
through $\x^{\textsc{u} p}$ and $\x^{\textsc{u} q}$ and the one through $\x^{\textsc{l} p}$ and
$\x^{\textsc{l} q}$. They have the same slope,
\[
\sigma = \frac{d_2}{d_1} = p\frac{\eta_2 - 1}{\eta_1 - 1} < 0.
\]
Then the two equations are $x_2 = \sigma x_1 +
(x_2^{\textsc{u} p} - \sigma x_1^{\textsc{u} p})$ and $x_2 = \sigma x_1 + (x_2^{\textsc{l} p} -
\sigma x_1^{\textsc{l} p})$. The former equation is used below to compute the
area of the remaining part of the cross-section.
The distance between two lines with equal slope $x_2 =
\sigma x_1 + r'$ and $x_2 = \sigma x_1 + r''$ is $\frac{\lvert r''
  - r'\rvert}{\sqrt{1 + \sigma^2}}$, therefore
\[
\textstyle \delta=\frac{\lvert(x_2^{\textsc{u} p} - \sigma x_1^{\textsc{u} p}) - (x_2^{\textsc{l} p} - \sigma x_1^{\textsc{l} p})\rvert}{\sqrt{1 + \sigma^2}}
      =\frac{\lvert(p -\sigma)x_1^{\textsc{u} p} - (p-\sigma)x_1^{\textsc{l} p}\rvert}{\sqrt{1+\sigma^2}}
      =\frac{p-\sigma}{\sqrt{1 + \sigma^2}}(x_1^{\textsc{l} p}-x_1^{\textsc{u} p}).
\]
Hence the area of the trapezoid is
\[
  A_1(z) = \frac{1}{2} \delta (\Delta_{\textsc{l}} + \Delta_{\textsc{u}}) =
  \frac{\tau(p-\sigma)}{2\sqrt{1 + \sigma^2}} ((x_1^{\textsc{l} p})^2 - (x_1^{\textsc{u} p})^2).
\]
Note that $\sqrt{1 + \sigma^2} = \sqrt{1 + p^2 \left(\frac{\eta_2 -
    1}{\eta_1 - 1}\right)^2} = \frac{1}{1-\eta_1}\sqrt{(\eta_1-1)^2 +
p^2(\eta_2-1)^2}$ and that $p-\sigma = p - p\frac{\eta_2 - 1}{\eta_1 -
  1} = p\frac{\eta_2 - \eta_1}{\eta_1 - 1}$, therefore
\[
\frac{\tau(p-\sigma)}{2\sqrt{1 + \sigma^2}} =
\frac{p}{2}(\eta_2 - \eta_1).
\]
In the remainder of this section, the calculation for the
cross-section area is divided for the two cases $\beta \ge 1$ and
$\beta \le 1$.

\paragraph*{Case 1: $\beta \ge 1$}

The coordinates $(x_1^{\textsc{l} p}, x_2^{\textsc{l} p})$ of
$L_p$ satisfy $\lambda(d_2 x_1^{\textsc{l} p} - d_1 x_2^{\textsc{l} p})^{\beta} =
x_1^{a_1}x_2^{a_2} = z$ (because $f'_\ell$ matches $f$ on
$P\cup Q$), whereas those of $U_p$ satisfy $z_0 + \left( \gamma
(x_1^{\textsc{u} p})^{a_1} (x_2^{\textsc{u} p})^{a_2} \right)^{\frac{1}{\beta}} =
z$; both satisfy $x_2 = p x_1$.
Therefore
\[
\begin{array}{rlll}
%
  (x^{\textsc{l} p}_1)^{a_1}(px^{\textsc{l} p}_1)^{a_2} = z
   &\Leftrightarrow &
  x_1^{\textsc{l} p} = (p^{-a_2}z)^{\frac{1}{\beta}}\\
  z_0 + \left( \gamma p^{a_2}(x_1^{\textsc{u} p})^{\beta} \right)^{\frac{1}{\beta}} = z
   &\Leftrightarrow &
  x_1^{\textsc{u} p} = \frac{z-z_0}{(\gamma p^{a_2})^{\frac{1}{\beta}}},\\
\end{array}
\]
and similarly for other points we obtain
\begin{equation}
  \label{eq:trapvertices}
  \begin{array}{lllll}
    x_1^{\textsc{l} p} = (p^{-a_2}z)^{\frac{1}{\beta}}, & x_2^{\textsc{l} p} = px_1^{\textsc{l} p}, & \quad &
    x_1^{\textsc{u} p} = \frac{z-z_0}{(\gamma p^{a_2})^{\frac{1}{\beta}}}, & x_2^{\textsc{u} p} = p x_2^{\textsc{u} p},\\
    x_1^{\textsc{l} q} = (q^{-a_2}z)^{\frac{1}{\beta}}, & x_2^{\textsc{l} q} = qx_1^{\textsc{l} q}, & &
    x_1^{\textsc{u} q} = \frac{z-z_0}{(\gamma q^{a_2})^{\frac{1}{\beta}}}, & x_2^{\textsc{u} q} = q x_2^{\textsc{u} q}.\\
  \end{array}
\end{equation}
The area of the trapezoid is hence
$
  A_1(z) = \frac{1}{2}p(\eta_2 - \eta_1)
  \left(z^{\frac{2}{\beta}} - (z - z_0)^2/\gamma^{\frac{2}{\beta}} \right)
$.

The area defined by arc and segment between $U_p$ and $U_q$
is equal to the integral of the difference between two
functions: the linear function through $U_p$ and $U_q$,
$x_2 = \sigma x_1 + (x_2^{\textsc{u} p} - \sigma x_1^{\textsc{u} p})$,
and the function $x_2 = \left((\frac{z - z_0}{\gamma})^\beta
x_1^{-a_1}\right)^{\frac{1}{a_2}}$ which arises from solving $z_0 +
\gamma(x_1^{a_1}x_2^{a_2})^{\frac{1}{\beta}} = z$ for $x_2$.

The difference is $\sigma x_1 + (p - \sigma) x_1^{\textsc{u} p} -
(\frac{z - z_0}{\gamma})^\frac{\beta}{a_2} x_1^{-\frac{a_1}{a_2}}$,
whose primitive is
\[
\left\{
  \begin{array}{ll}
    \frac{1}{2}\sigma x_1^2 +
    (p - \sigma) x_1^{\textsc{u} p} x_1 -
    \frac{a_2}{a_2 - a_1}(\frac{z - z_0}{\gamma})^\frac{\beta}{a_2}
    x_1^{\frac{a_2 - a_1}{a_2}} & \textrm{ if } a_1 \neq a_2\\
    \frac{1}{2}\sigma x_1^2 +
    (p - \sigma) x_1^{\textsc{u} p} x_1 -
    (\frac{z - z_0}{\gamma})^\frac{\beta}{a_2}
    \log x_1 & \textrm{ otherwise},\\
  \end{array}
\right.
\]
which yields, for $a_1 \neq a_2$,
\[
\begin{array}{rcl}
  A_2(z) = &\frac{1}{2}\sigma ((x_1^{\textsc{u} p})^2 - (x_1^{\textsc{u} q})^2) +
  (p-\sigma) (x_1^{\textsc{u} p} - x_1^{\textsc{u} q}) -\\
     & \frac{a_2}{a_2 - a_1}(\frac{z -
    z_0}{\gamma})^\frac{\beta}{a_2}
  ((x_1^{\textsc{u} p})^{\frac{a_2 - a_1}{a_2}} -
   (x_1^{\textsc{u} q})^{\frac{a_2 - a_1}{a_2}}),
\end{array}
\]
and $ A_2(z) = \frac{1}{2}\sigma ((x_1^{\textsc{u} p})^2 - (x_1^{\textsc{u} q})^2) +
(p-\sigma) (x_1^{\textsc{u} p} - x_1^{\textsc{u} q}) - (\frac{z -
  z_0}{\gamma})^\frac{\beta}{a_2} (\log x_1^{\textsc{u} p} - \log x_1^{\textsc{u} q})$
otherwise.

The total area of the cross-section is then a polynomial function with
terms $z^{\frac{2}{\beta}}$, $(z-z_0)^2$, and $(z -
z_0)^\frac{\beta}{a_2}$. For simplicity we write it as $A(z) = b_1
z^{\frac{2}{\beta}} + b_2 (z - z_0)^2 + b_3 (z -
z_0)^{\frac{\beta}{a_2}}$ for opportune values of $b_1, b_2, b_3$
which depend on $a_1, a_2, p, q, \ell, u$. Therefore the volume of the
convex hull for $n=2$ is $\textrm{Vol}(\textrm{conv}(X \cap W_{12})) =
\int_{\ell}^{u} A(z)dz$, i.e.,
\[
\begin{array}{ll}
&\left[
b_1\frac{\beta}{\beta+2}z^{1+\frac{2}{\beta}} + \frac{1}{3} b_2(z-z_0)^3 +\frac{a_2}{a_2 + \beta}b_3(z-z_0)^{1 + \frac{\beta}{a_2}}
\right]_\ell^u\\
=&b_1\frac{\beta}{\beta+2}(u^{1+\frac{2}{\beta}} - \ell^{1+\frac{2}{\beta}}) +
 \frac{1}{3} b_2\left((u-z_0)^3 - (\ell-z_0)^3\right) +\\
 &\frac{a_2}{a_2+\beta}b_3\left((u-z_0)^{1 + \frac{\beta}{a_2}} - (\ell-z_0)^{1 + \frac{\beta}{a_2}}\right),
\end{array}
\]
where $b_3$ is defined differently depending on whether $a_1 = a_2$ or
not.

\paragraph*{Case 2: $\beta \le 1$}

For $L_p$ one must solve the system $x_2^{lp} = p x_1^{lp}, z =
\zeta(d_2 x_1^{lp} - d_1 x_2^{lp}) + z_0$ for $x_1^{lp}$, which yields
$x_1^{lp} = \frac{z - z_0}{(d_2 - d_1 p) \zeta} =
\frac{(z - z_0) \big(u^{\frac{1}{\beta}} - \ell^{\frac{1}{\beta}}\big)}{(u -
  \ell) q^{\frac{a_2}{\beta}}}$. For $U_p$ and $U_q$,
note that $f$ constitutes the upper envelope, hence $(x_1^{\textsc{u} p})^{a_1}
(x_2^{\textsc{u} p})^{a_2} = z$ and $x_2^{\textsc{u} p} = p x_1^{\textsc{u} p}$ yields $x_1^{\textsc{u} p} =
z^{\frac{1}{\beta}}p^{-a_2}$.
Then the area of the trapezoid specializes to
\[
  A_1(z) = \frac{p}{2}(\eta_2 - \eta_1) \left( \left(\frac{(z - z_0)
    \big(u^{\frac{1}{\beta}} - \ell^{\frac{1}{\beta}}\big)}{(u -
  \ell) q^{\frac{a_2}{\beta}}}\right)^2 - (z^{\frac{1}{\beta}}p^{-a_2})^2\right).
\]

Finally, the curve delimiting $X\cap
W_{12}$ from below has equation $x_2 =
z^{\frac{1}{a_2}} x_1^{-\frac{a_1}{a_2}}$, while the line between
$\x^{\textsc{u} q}$ and $\x^{\textsc{u} p}$ is $x_2 = x_2^{\textsc{u} p} + \sigma(x_1 - x_1^{\textsc{u} p}) =
\sigma x_1 + x_1^{\textsc{u} p}(p - \sigma)  =
 \sigma x_1 + (p - \sigma) z^{\frac{1}{\beta}} p^{-a_2}$. The area of
$\textrm{conv}(C_\ell \cap W_{12})$ is then the integral, between
$x_1^{\textsc{u} q}$ and $x_1^{\textsc{u} p}$, of the function $ \sigma x_1 + (p - \sigma)
 z^{\frac{1}{\beta}} p^{-a_2} - z^{\frac{1}{a_2}}
 x_1^{-\frac{a_1}{a_2}}$, whose primitive is similar, in structure, to
 that for the case $\beta \ge 1$, and thus we omit it here. The volume
 of the convex hull for $\beta \le 1$ is therefore an integral of a
 polynomial in $z$ with the same exponents as for $\beta \ge 1$, but with different
 coefficients due to different coordinates of $L_p, L_q, U_p, U_q$.





\section{Concluding remarks and open questions}
\label{sec:conclusions}

We considered the monomial $f(\x)=\prod_{k\in N}x_k^{a_k}$, with
positive exponents, for $\x \in
\Rnp$. We proved that its upper envelope, for general $n\ge 2$, is a
conic function when the domain of $f$ restrict its value between
$\ell\ge 0$ and $u$. The result holds even when two variables
$x_i,x_j$, apart from being non-negative, are also constrained to form a
linear cone $p x_i \le x_j \le q x_i$ pointed at the origin.

For $n=2$ and when maintaining the linear cone $p x_1 \le x_2 \le q
x_1$, we also find the lower envelope of $f(\x)$ and thus are able to
describe the convex hull of $F(W_{12})$ and its volume.

A seemingly easy extension that we have not considered here is the
case with $n=2, a_1, a_2 < 0$, which makes $f$ convex and perhaps
yields a convex hull similar to that of the case $\beta \le 1$.

As discussed in Section \ref{sec:intro}, these results are applicable
to either problems whose model has constraints such as $p x_i \le x_j
\le q x_i$ or algorithms that enforce such constraints as cutting
planes or branching operations. If such {\em wedge} constraints are
not a natural way to describe a problem or an algorithm, the obvious
tradeoff is between two approaches for MINLO:
\begin{itemize}
  \item the classic approach, where variables have initial lower and
    upper bounds and branching rules are of the form $x_k \le b \vee
    x_k \ge b$, but
    the convex hull of $\{(x_1,x_2)\in \mathbb R_+^2:
    x_1^{a_1}x_2^{a_2}\in [\ell,u], x_1\in [\ell_1, u_1],
    x_2\in [\ell_2, u_2]\}$ is not known and therefore a decomposition
    like the one described in Section \ref{sec:intro} is needed;
  \item an approach where the conic constraint $p x_1 \le x_2 \le q
    x_1$ is enforced and the convex hull is known.
\end{itemize}
Such tradeoff can make one or the other relaxation tighter.
Computational tests are needed to assess the quality of
the convex hull above, especially because the partitioning of $\mathbb
R_+^2$ through linear inequalities is not as standard as variable
bounds. 

\bibliography{minlp}
\bibliographystyle{abbrvnat}

\end{document}